\documentclass[10pt,leqno]{amsart}
\usepackage{titlesec}
\titleformat{\subsection}
{\normalfont\bfseries}
{\thesubsection}
{1em}
{}
\titleformat{\section}
{\normalfont\Large\bfseries}
{\thesection\quad}
{0pt}
{}

\usepackage{epstopdf}
\usepackage[caption=false]{subfig}
\usepackage{mathtools}
\usepackage{float}
\usepackage{graphicx}
\usepackage{xcolor}

\usepackage{amsmath,amsfonts,amssymb}
\usepackage{bbm}
\usepackage{algorithm}
\usepackage{algpseudocode}
\usepackage{changes}

\definecolor{Tommaso}{rgb}{0, 0, 1}

\definecolor{Niklas}{rgb}{0, 0.5, 0.5}

\definecolor{Philipp}{rgb}{0, 0.39, 0}

\definecolor{David}{rgb}{0.55, 0.71, 0}

\newcommand{\LQ}[0]{\left[}
\newcommand{\RQ}[0]{\right]}
\newcommand{\EE}[0]{\mathbb{E}}
\newcommand{\PP}[0]{\mathbb{P}}

\DeclareMathOperator{\boxN}{\mathbf{\Box}^{\mathbb{N}}}
\newcommand{\bsxi}{{\boldsymbol{\xi}}}
\newcommand{\xb}{{\boldsymbol{x}}}

\newcommand{\abs}[1]{\left| #1 \right|}
\newcommand{\babs}[1]{\big| #1 \big|}

\newcommand{\norm}[1]{\left\Vert #1 \right\Vert}
\newcommand{\tnorm}[1]{\Vert #1 \Vert}
\newcommand{\bnorm}[1]{\big\Vert #1 \big\Vert}

\newcommand{\set}[1]{\left\lbrace #1 \right\rbrace}
\newcommand{\tset}[1]{\lbrace #1 \rbrace}

\theoremstyle{plain}
\newtheorem{theorem}{Theorem}[section]
\newtheorem{lemma}[theorem]{Lemma}
\newtheorem{corollary}[theorem]{Corollary}
\newtheorem{proposition}[theorem]{Proposition}

\theoremstyle{definition}

\theoremstyle{remark}

\usepackage[breaklinks=true,colorlinks=true,linkcolor=blue,citecolor=blue,urlcolor=cyan,pdfborder={0 0 0}]{hyperref}

\makeatletter
\renewcommand\paragraph[1]{\@startsection{paragraph}{4}{\parindent}%
{0.1ex \@plus 0.1ex \@minus 0.2ex}
{-1em}
{\normalfont\itshape\normalsize}*{#1.\quad}}
\makeatother

\hypersetup{
    colorlinks=true,
    linkcolor=blue,
    filecolor=blue,
    urlcolor=blue,
    citecolor=blue
}

\begin{document}
    \title[MLSGD for Risk-Averse PDE-Constrained Optimization]
    {\Large Multilevel Stochastic Gradient Descent for \\ Risk-Averse PDE-Constrained Optimization}
    \author{Niklas Baumgarten}
    \author{Philipp A.~Guth}
    \author{David Schneiderhan}
    \author{Tommaso Vanzan}
    \date{\today}

    \maketitle

    \let\thefootnote\relax\footnotetext{
        \noindent\texttt{niklas.baumgarten@uni-heidelberg.de}; Heidelberg University \\
        \indent\texttt{philipp.guth@ricam.oeaw.ac.at}; Johann Radon Institute for Computational and Applied Mathematics \\
        \indent\texttt{david.schneiderhan@kit.edu}; Karlsruhe Institute of Technology \\
        \indent\texttt{tommaso.vanzan@polito.it}; Politecnico di Torino
    }

    \begin{abstract}
    We present recent advances in applying and analyzing
    multilevel stochastic gradient descent algorithms to
    risk-averse, three-dimensional PDE-constrained optimization problems.
    The algorithm uses adaptive multilevel Monte Carlo gradient estimates,
    provides parallel scalability as well as improved convergence rates
    and computational complexity compared to
    standard batched stochastic gradient descent methods.
    We study the method in computationally demanding settings
    using three-dimensional elliptic diffusion problems
    and large risk-aversion parameters.
\end{abstract}
    \section{Introduction}
\label{sec:introduction}

Optimization problems governed by partial differential equations (PDEs)
under uncertainty arise in a wide range of applications,
including engineering design, geophysics, and medical applications.
In such settings, uncertainties in model parameters,
boundary conditions, or external forcing
propagate through the governing equations and significantly affect the system response.
As a consequence, optimization based solely on a single realization/estimate
or on expected values may lead to solutions that are highly
sensitive to rare but impactful events.
This has motivated the development of risk-averse formulations
for PDE-constrained optimization problems,
in which suitable risk measures are incorporated
into the objective functional to penalize undesirable outcomes.
Among the various approaches to quantify risk, coherent
and convex risk measures have received considerable attention
(see~\cite{heinkenschloss2025optimization} for a recent overview).
In particular, the entropic risk measure
provides a monotonic transition between
the expectation and worst-case behavior, controlled by a risk-aversion parameter.
Furthermore, for each fixed parameter, it is strictly convex and smooth.

The resulting optimization problems, even in the risk-neutral case,
are computationally demanding, as they involve high-dimensional integration
over the underlying probability space combined with the repeated solution of PDEs.
Beyond sample average approximations~\cite{kouri2013trust,guth2021quasi,guth2024parabolic,doi:10.1137/22M1512636,doi:10.1137/22M1532263},
a widely used class of methods for such problems are stochastic gradient
methods (SGD) ~\cite{geiersbach2019projected, geiersbach2020stochastic, martin2021complexity},
which approximate gradients of the objective functional using a single realization or small batches drawn at each iteration. More recently, variants that adapt batch sizes throughout the optimization process have gained significant attention \cite{bollapragada2018adaptive,beiser2023adaptive}.
However, despite avoiding discretization errors arising from a finite representation of the sample space, the efficiency of these methods is often limited by the high variance of gradient estimators,
especially in a risk-averse setting.
To mitigate these issues, variance reduction techniques are essential see,
e.g.,~\cite{pieraccini2025adaptive,nobile2025stochasticgradientleastsquarescontrol}.

In this work, we adapt and analyze an adaptive multilevel stochastic gradient descent (MLSGD) method,
as introduced in~\cite{baumgarten2025multilevel},
for PDE-constrained optimization problems involving the entropic risk-measure.
The approach combines an adaptive stochastic gradient method with
multilevel Monte Carlo (MLMC) techniques~\cite{Giles_2015},
which exploit a hierarchy of discretizations to achieve significant
variance reduction at an optimal computational cost.
MLMC methods are well established for forward and inverse uncertainty
quantification~\cite{cliffe2011multilevel,teckentrup2013further,barth2011multi,giles2016multilevel}
and have been successfully applied to risk-neutral optimal control
problems~\cite{ali2017multilevel,ciaramella2024multigrid,van2019robust,guth2023multilevel,nobile2025multilevel},
as well as to the risk-averse settings. In particular, \cite{ganesh2023gradient}
presents an adaptive stochastic gradient algorithm based on MLMC gradient estimates for the conditional value-at-risk, which is another popular risk-measure.

Building on the aforementioned works, the main contributions of this paper are threefold:
First, we propose a MLSGD algorithm based on multilevel
gradient estimators and adaptive batch sizes tailored to the entropic risk measure.
At each optimization step, the method dynamically balances discretization
and sampling errors, selecting optimal batch sizes and discretization levels.
To achieve this, we investigate the influence of
the risk-aversion parameter on the algorithm and its adaptivity.

 Second, we provide a rigorous convergence analysis of the proposed method.
Our study shows that, under suitable assumptions,
the algorithm converges linearly in the number of optimization steps and in expectation.
In addition, we discuss the algorithmic complexity and show that an idealized variant of our algorithm achieves
the same asymptotic complexity as MLMC applied to forward problems,
up to constants depending on the risk-aversion parameter.

 Third, we address the significant computational challenges
arising from large-scale three-dimensional PDEs and high risk-aversion parameters.
To this end, we present an efficient parallel implementation that
combines sample-wise parallelism on coarse levels with
domain decomposition techniques on fine levels.
We build upon the distributed multilevel data structure
proposed in~\cite{baumgarten2025budgeted} and improve the memory layout,
enabling scalability also on three spatial dimensions
and allowing for the treatment of large batch sizes
and high-fidelity discretizations
encountered in practical applications.

Our theoretical findings are supported by numerical
experiments for three-dimensional elliptic diffusion
problems with random coefficients.
These experiments confirm the predicted convergence
rates and demonstrate the robustness and efficiency
of the method across different regimes of risk-aversion.

The remainder of the paper is organized as follows.
In Section~\ref{sec:risk-averse-pde-constrained-optimization},
we introduce the risk-averse PDE-constrained optimization problem
and recall properties of the entropic risk measure.
Section~\ref{sec:convergence-analysis} is devoted to
the convergence and complexity analysis of the MLSGD method.
In Section~\ref{sec:algorithmic-and-implementation-details},
we discuss algorithmic aspects and present the parallel implementation.
Numerical results are reported in Section~\ref{sec:numerical-experiments},
followed by concluding remarks and perspectives for future work
in Section~\ref{sec:outlook-and-conclusion}.

    \section{Risk-Averse PDE-constrained Optimization}
\label{sec:risk-averse-pde-constrained-optimization}

We start by presenting the PDE-constrained optimization problem,
introducing the entropic risk measure, and outlining the optimization algorithm.
To this end, let~$H$ be a separable Hilbert space representing a state space,
and assume that it serves as a pivot space, that is,
we will identify it with its topological dual,~$H = H'$.
Furthermore, let~$V$ be another separable Hilbert space that is continuously
and densely embedded in~$H$, leading to the Gelfand triplet~$V \hookrightarrow H \hookrightarrow V'$.
In addition, let $U = U'$ be a third separable Hilbert space,
and we denote by $U_{\mathrm{ad}}$ a nonempty, closed, convex and bounded subset of $U$.
As an example, one may consider a Lipschitz domain $\mathcal{D}$
and the Sobolev spaces $V=H^1_0(\mathcal{D})$, $H=L^2(\mathcal{D})$
and $U=L^2(\Gamma_N)$, $\Gamma_N$ being a subset of the boundary of $\mathcal{D}$.

In this work, we aim at finding a deterministic control~$u \in U_{\mathrm{ad}}$
that steers a random state~$y$ as close as possible to a desired target~$g \in H$
while minimizing a given cost functional~$J$ involving a risk measure~$\mathcal{R}$.
Motivated by series expansions of random fields,
we assume the uncertainty of the problem is modeled through
a countable sequence~$\bsxi = (\xi_j)_{j\in \mathbb{N}}$ of independent
and identically distributed uniform random variables~$\xi_j \sim \mathcal{U}([-1,1])$,
i.e.,~$\bsxi \in [-1,1]^{\mathbb{N}} \eqqcolon \boxN$,
defined on a complete probability space $(\Omega,\mathcal{F},\PP)$.
We shall be interested in continuous functions~$\bsxi \mapsto Z(\bsxi)$ taking
values in some separable Hilbert space.
In this case~$Z$ is~$\mathrm d\bsxi$-Bochner-integrable,
and we can write expected values as
\begin{align*}
    \mathbb{E}[Z] \coloneqq \int_{\boxN} Z(\bsxi) \,\mathrm d\bsxi,
    \quad \text{where} \quad
    \mathrm{d}\bsxi \coloneqq \bigotimes\limits_{j\in \mathbb{N}} \frac{\mathrm d\xi_j}{2}.
\end{align*}
Our model problem is the minimization of the objective functional
\begin{align}
    \label{eq:cost}
    J(u,y) \coloneqq \mathcal{R}(\Phi(y))
    + \frac{\lambda}{2} \|u\|_U^2,
    \quad  \text{where} \quad
    \Phi(y) \coloneqq \frac{1}{2}\|y-g\|_H^2,
\end{align}
for~$\lambda>0$, subject to
\begin{align}
    \label{eq:sys}
    &a(\bsxi;y,v)= \langle Bu,v \rangle_{V',V},
    \qquad \forall v\in V,
    \text{ and for a.e. }
    \bsxi \in \boxN, \\
    &u \in U_{\mathrm{ad}}.
    \label{eq:boxconstr}
\end{align}
We assume that~$\{a(\bsxi;\cdot,\cdot)\mid \bsxi \in \boxN\}$ forms a
family of parameter-dependent,
continuous and uniformly coercive bilinear forms,
i.e., there exists a $a_{\min} >0$ such that $a(\bsxi; v,v) \ge a_{\min} \|v\|^2_V$
for all~$\bsxi \in \boxN$.
Thus, with the family of bilinear forms we can associate a family of
operators~$\{A(\bsxi) \in \mathcal{L}(V,V') \mid \bsxi \in \boxN\}$,
which is a family of isomorphisms with uniformly
bounded inverses~$\|A(\bsxi)^{-1}\|_{\mathcal{L}(V',V)} \le \frac{1}{a_{\min}}$
for all~$\bsxi \in \boxN$.
Hence, given a control operator~$B \in \mathcal{L}(U,V')$ and~$u \in U$,
there exists a unique solution~$y_u(\bsxi) \in V$ of~\eqref{eq:sys}
for a.e.~$\bsxi \in \boxN$.
This permits to define the family
of parametric control-to-state operators
$\{S[\bsxi]\in \mathcal{L}(U,V)\mid \bsxi\in \boxN\}$ such that
\begin{align*}
    S[\bsxi](u) \coloneqq A(\bsxi)^{-1}B u,
\end{align*}
for~$u\in U_{\mathrm{ad}}$ and~$\bsxi \in \boxN$.
Under these assumptions the mapping~$\bsxi \mapsto y(\bsxi) \in V$
is continuous if~$\bsxi \mapsto A(\bsxi) \in \mathcal{L}(V,V')$ is continuous.
Consequently, for every $u\in U_{\mathrm{ad}}$, the misfit $\Phi(S[\cdot](u))$
is an essentially bounded random variable
mapped to the real line by the risk measure $\mathcal{R}$.

Using the control-to-state operator,
we can formulate the reduced problem as
\begin{align}
    \label{eq:J_red}
    \min_{u \in U_{\mathrm{ad}}} \mathcal{J}(u),
    \quad \text{where} \quad
    \mathcal{J}(u) = J(u,S[\cdot](u)).
\end{align}

Provided that $\mathcal{R}$ is a proper, closed,
convex and monotonic risk-measure, the existence
of a solution to~$\min_{u \in U_{\mathrm{ad}}} \mathcal{J}(u)$
(and thus~\eqref{eq:cost} subject to~\eqref{eq:sys}--\eqref{eq:boxconstr})
follows from~\cite[Prop.~3.12]{KouriSurowiec1}.
Uniqueness follows from the strong convexity of~$\mathcal{J}$
in~\eqref{eq:J_red} due the penalization term.
For sufficiently regular~$\mathcal{J}$,
since $U_{\mathrm{ad}}$ is a nonempty and convex set
(see, e.g.~\cite[Lemma 2.21]{frediTr}),
the optimal control~$u^\star$ satisfies the variational inequality
\begin{equation}
    \label{eq:variational_inequality}
    \langle \nabla \mathcal{J}(u^\star), u - u^\star \rangle_{U} \ge 0 \quad \forall u \in U_{\mathrm{ad}},
\end{equation}
where~$\nabla \mathcal{J}$ denotes the gradient of~$\mathcal{J}$,
i.e., the Riesz representer in~$U$ of the Fr\'echet
derivative of~$\mathcal{J}$.
Moreover, for convex objective functionals,
like~$\mathcal{J}$ in~\eqref{eq:J_red},
the variational inequality is a necessary and sufficient condition for optimality.
Since $U_{\mathrm{ad}}$ is closed and convex,
\eqref{eq:variational_inequality} can be equivalently formulated
(\cite[Lemma 1.11]{hinze2008optimization}) for any $\tau>0$ as
\begin{align*}
    u^\star-\Pi_{U_\mathrm{ad}}(u^\star-\tau \nabla \mathcal{J}(u^\star))=0,
\end{align*}
where $\Pi_{U_\mathrm{ad}}:U\rightarrow U_{\mathrm{ad}}$
is the projector onto the admissible set of controls.

    \subsection{Entropic risk measure}
\label{subsec:entropic-risk-measure}

In this work, we focus on the \textit{entropic risk measure}
with risk-aversion parameter~$\theta>0$ which,
for a random variable~$Z \in L^\infty(\boxN;\mathbb{R})$, is defined as
\begin{equation}
    \label{eq:entropic-risk}
    \mathcal{R}_\theta (Z) \coloneqq \tfrac{1}{\theta} \log{\left( \mathbb{E} \left[ \mathrm \exp(\theta Z)\right]\right)}.
\end{equation}
As a function of~$\theta$, the entropic risk measure is
increasing and strictly increasing if~$Z$ is not constant (a.s.).
Moreover, it can be shown (cf.~\cite[Eqn.~29]{FoellmerKnispel}) that
\begin{align*}
    \lim_{\theta \to 0} \mathcal{R}_\theta(Z) = \mathbb{E}\left[Z\right]
    \qquad \text{and} \qquad
    \lim_{\theta \to \infty} \mathcal{R}_\theta(Z) = {\mathrm{ess}\,\sup}(Z),
\end{align*}
hence, the scalar parameter $\theta$ controls the level of risk-aversion.
A direct calculation shows that the gradient of $\mathcal{J}$,
when $\mathcal{R}$ coincides with the entropic risk measure, is
\begin{equation}
    \label{eq:exact-gradient}
    \nabla \mathcal{J}(u) = \frac{\mathbb{E} [\exp(\theta \Phi(y)) \, B^\ast q]}{\mathbb{E} [\exp(\theta \Phi(y))]} + \lambda u
    = \frac{\mathbb{E} [X]}{\mathbb{E} [Y]} + \lambda u
\end{equation}
where $y$ and $q$ are the solutions of
\begin{align*}
    \langle A(\bsxi) y(\bsxi), v\rangle_{V',V} &= \langle Bu, v\rangle_{V',V} && \forall v\in V,
    \quad  \text{a.e. } \bsxi \in \boxN,\\
    \langle A^\ast(\bsxi) q(\bsxi), v\rangle_{V',V} &= \langle y(\bsxi) - g, v\rangle_{H} &&\forall v\in V,
    \quad \text{a.e. } \bsxi \in \boxN
\end{align*}
and
\begin{equation}\label{eq:definition_X_Y}
    X \coloneqq \exp(\theta \Phi(y))B^\ast q,\qquad Y \coloneqq \exp(\theta \Phi(y))
\end{equation}
are, respectively, $U$-valued and real-valued random variables introduced to ease notation.

    \subsection{Optimization method}
\label{subsec:optimization-method}

To compute a numerical approximation of $u^\star$,
we propose and analyze an adaptive sampling stochastic
algorithm based on a MLMC gradient estimation.
Specifically, we consider a sequence of finite element (FE)
subspaces $\{V_{\ell}\}_{\ell \in\mathbb{N}} \subset V$ of increasing dimension,
and the FE solution operators $S_{\ell}[\bsxi]:U\rightarrow V_{\ell}$ so that
\begin{align*}
    a(\bsxi;S_{\ell}[\bsxi](u),v_\ell)=\langle B u, v_{\ell}\rangle_{V',V},
    \quad\forall v_{\ell}\in V_{\ell}, \text{ and for a.e. } \bsxi \in \boxN.
\end{align*}
For every $\ell\in \mathbb{N}$, we define $y_\ell \coloneqq S_{\ell}[\cdot](u)$ and $q_{\ell}$ as the solution of
\begin{align*}
    a(\bsxi;v_\ell,q_{\ell})=\langle S_\ell[\bsxi](u)-g, v_{\ell}\rangle_{H},
    \quad \forall v_{\ell}\in V_{\ell}, \text{ and for a.e. } \bsxi \in \boxN.
\end{align*}
Note that we omit the explicit dependence of both
$y_\ell$ and $q_{\ell}$ on the control variable $u$,
as this will be clear from the context.

Next, for a given $L \in \mathbb{N}$ and a set of sample sizes $\set{M_\ell}_{\ell=1}^L$,
we introduce the gradient estimator
\begin{equation}
    \label{eq:gradient_estimator}
    \nabla \mathcal{J}^{\mathrm{ML}}(u)
    \coloneqq
    \frac{\EE^{\mathrm{ML}}\big[\exp(\theta \Phi(y_L)) \, B^\ast q_L \big]}
    {\EE_{\geq 1}^{\mathrm{ML}}\big[\exp(\theta \Phi(y_L))\big]} + \lambda u
    = \frac{\EE^{\mathrm{ML}}\big[X_L\big]}{\EE_{\geq 1}^{\mathrm{ML}}\big[Y_L\big]} + \lambda u,
\end{equation}
where, for every $\ell\in \mathbb{N}$,
\begin{equation}
    \label{eq:definition_Xell_Yell}
    X_\ell \coloneqq \exp(\theta \Phi(y_\ell))B^\ast q_\ell,\quad  Y_\ell \coloneqq \exp(\theta \Phi(y_\ell)).
\end{equation}
The MLMC estimators appearing in~\eqref{eq:gradient_estimator} are given by
\begin{align}
    \label{eq:ml_estimator_numerator}
    \EE^{\mathrm{ML}}\big[X_L\big]
    &=\sum_{\ell=1}^L \frac{1}{M_{\ell}}\sum_{m=1}^{M_\ell} X^{(m)}_{\ell}-X^{(m)}_{\ell-1}
    = \sum_{\ell=1}^L \frac{1}{M_{\ell}}\sum_{m=1}^{M_\ell} \Delta X^{(m)}_{\ell}, \\
    \label{eq:ml_estimator_denominator}
    \EE^{\mathrm{ML}}\big[Y_L\big]
    &=\sum_{\ell=1}^L \frac{1}{M_{\ell}}\sum_{m=1}^{M_\ell} Y^{(m)}_{\ell}-Y^{(m)}_{\ell-1}
    = \sum_{\ell=1}^L \frac{1}{M_{\ell}}\sum_{m=1}^{M_\ell} \Delta Y^{(m)}_{\ell},
\end{align}
with $\EE_{\geq 1}^{\textrm{ML}}\left[Y_L\right] \coloneqq \max\{1,\EE^{\mathrm{ML}}\big[Y_L\big]\}$.
Indeed, although $\EE[Y]\geq 1$, the MLMC estimator $\EE^{\mathrm{ML}}\left[Y_L\right]$ is a random variable,
and has a nonzero probability of attaining values below one,
potentially even arbitrarily close to zero due to the difference of hierarchical contributions.
The truncation guarantees that the denominator in~\eqref{eq:gradient_estimator}
remains uniformly bounded away from zero.
While none of our numerical experiments produced an instance of $\EE^{\mathrm{ML}}[Y_L] < 1$,
we include the maximum operator to perform a rigorous theoretical analysis of the algorithm.

In \eqref{eq:ml_estimator_numerator} and \eqref{eq:ml_estimator_denominator},
we have used the shorthand notation
\begin{equation}
    \label{eq:shorthand-notation}
    \begin{aligned}
        X^{(m)}_{\ell}&\coloneqq\exp(\theta \Phi[\bsxi^{(m,\ell)}](S_\ell[\bsxi^{(m,\ell)}](u))) B^\ast q_\ell^{(m,\ell)},
        \,\,\,\,
        Y^{(m)}_{\ell-1}\coloneqq\exp(\theta \Phi[\bsxi^{(m,\ell)}](S_{\ell-1}[\bsxi^{(m,\ell)}](u))),
        \\
        X^{(m)}_{\ell-1}&\coloneqq\exp(\theta \Phi[\bsxi^{(m,\ell)}](S_{\ell-1}[\bsxi^{(m,\ell)}](u))) B^\ast q_{\ell-1}^{(m,\ell)},
        \,\,\,\,
        Y^{(m)}_{\ell}\coloneqq\exp(\theta \Phi[\bsxi^{(m,\ell)}](S_{\ell}[\bsxi^{(m,\ell)}](u))),
    \end{aligned}
\end{equation}
to emphasize that all quantities in~\eqref{eq:shorthand-notation}
depend on the random realization $\bsxi^{(m,\ell)}$, and set
\begin{equation}
    \label{eq:delta-notation}
    \Delta X^{(m)}_{\ell} \coloneqq X^{(m)}_{\ell}-X^{(m)}_{\ell-1}
    \quad \text{and} \quad
    \Delta Y^{(m)}_{\ell} \coloneqq Y^{(m)}_{\ell}-Y^{(m)}_{\ell-1}
\end{equation}
as well as used the convention that $X^{(m)}_0=0$ and $Y^{(m)}_0=0$.
Furthermore, for $\ell \in \set{1,\dots,L}$, the elements of $\{\bsxi^{(m,\ell)}\}_{m=1}^{M_\ell}$
are independent and identically distributed, and these sequences
are mutually independent and identically distributed across $\ell$.
For every $\ell=1,\dots,L$, $q_{\ell}^{(m,\ell)} \in V_\ell$ and $q_{\ell-1}^{(m,\ell)} \in V_{\ell-1}$
are the discrete adjoint variables, solutions of
\begin{align*}
    a(\bsxi^{(m, \ell)};v_\ell,q_{\ell}^{(m,\ell)})&=\langle S_\ell[\bsxi^{(m,\ell)}](u)-g, v_{\ell}\rangle_{H},
    \quad \forall v_{\ell}\in V_{\ell}, \;m=1,\dots,M_{\ell},\\
    a(\bsxi^{(m, \ell)};v_{\ell-1},q_{\ell-1}^{(m,\ell)})&=\langle S_{\ell-1}[\bsxi^{(m,\ell)}](u)-g, v_{\ell-1}\rangle_{H},
    \quad \forall v_{\ell-1}\in V_{\ell-1}, \;m=1,\dots,M_{\ell}.
\end{align*}
Hence, in \eqref{eq:shorthand-notation} the superscript $(m,\ell)$
refers to a specific realization of the random variable $\bsxi$,
while the subscript $\ell$ specifies the level of mesh refinement
used in the discretization of the state and adjoint equations.

Our optimization algorithm starts from an initial guess $u_0$ and consists in the iteration
\begin{equation}
    \label{eq:gradientdescentMLMC}
    u_{k+1}=\Pi_{U_{\text{ad}}}\left(u_{k}-\tau \nabla \mathcal{J}^{\mathrm{ML}}(u_k)\right).
\end{equation}
In the following, we aim at both determining a set of sufficient conditions such that
the sequence $\{u_k\}_{k\in\mathbb{N}}$ converges to $u^\star$,
and at testing the algorithm's efficiency in a massive parallel environment.
    \section{Convergence analysis}
\label{sec:convergence-analysis}
To achieve our first goal of establishing
sufficient convergence conditions,
we start by analyzing an adaptive stochastic gradient descent
algorithm with an MLMC gradient estimator.
For this analysis, building on the results of~\cite{beiser2023adaptive,ganesh2023gradient},
we assume that:
\begin{enumerate}
    \item[a)] $\mathcal{J}:U\rightarrow \mathbb{R}$ is continuously Fr\'echet differentiable.
    \item[b)] Lipschitz gradients: $\exists\;L\in\mathbb{R}^+$ such that for every $u_1,u_2\in U_{\mathrm{ad}}$,
    \begin{equation}
        \label{eq:lipschitz_gradients}
        \|\nabla \mathcal{J} (u_1)-\nabla \mathcal{J} (u_2)\|_U\leq L\|u_1-u_2\|_U.
    \end{equation}
    \item[c)] Strong convexity: $\exists\;c\in\mathbb{R}^+$ such that for every $u_1,u_2\in U_{\mathrm{ad}}$,
    \begin{equation}
        \label{eq:strongcoercivity}
        \langle\nabla \mathcal{J}(u_1)-\nabla \mathcal{J}(u_2),u_1-u_2\rangle_U \geq c \|u_1-u_2\|_U^2.
    \end{equation}
\end{enumerate}

The model problem~\eqref{eq:cost} s.t.~\eqref{eq:sys}--\eqref{eq:boxconstr} satisfies these assumptions.
To verify this, we first note that the map $u\mapsto \mathcal{J}(u)$ is a composition of smooth maps and, in particular, Fr\'echet differentiable with a continuous gradient. In addition, since $U_{\mathrm{ad}}$ is bounded and the family
$\set{A(\bsxi)}_{\bsxi \in \boxN}$ has uniformly bounded inverses,
both $\Phi(S[\bsxi](u))$ and $\exp(\theta \Phi(S[\bsxi](u))$ are bounded on
$U_{\mathrm{ad}}$ for all $\bsxi \in \boxN$.
Moreover, the boundedness of $\Phi(S[\bsxi](u))$ implies Lipschitz
continuity of $\exp(\theta \Phi(S[\bsxi](u)))$ on $U_{\mathrm{ad}}$
with a constant independent of~$\bsxi$.
Then, a direct calculation shows that the gradient \eqref{eq:exact-gradient} is Lipschitz.
Finally, the strong convexity of $\mathcal{J}$ follows from the convexity of $u\mapsto\Phi(S[\cdot](u)))$,
the convexity and monotonicity of~$\mathcal{R}$ and the strong convexity of $\|u\|_U^2$.

Next, we define the reduced projected gradient
\begin{align*}
    \mathcal{R}_{\tau}(u) \coloneqq \frac{1}{\tau}\left(u-\Pi_{U_{\text{ad}}}\left(u-\tau\nabla \mathcal{J}(u)\right)\right),
\end{align*}
and recall that under the previous assumptions,
if $\tau \in (0,\frac{1}{L})$ it holds that
(see, e.g.~\cite[Corollary 2.3.2]{nesterov2018lectures} and~\cite[Eqn.~(2.34)]{beiser2023adaptive})
\begin{equation}
    \label{eq:boundreducedgradient}
    \|\mathcal{R}_{\tau}(u)\|_U\leq \frac{1}{\tau}\left(1+\sqrt{1-c\tau}\right)\|u-u^\star\|_U,
\end{equation}
for every $u\in U_{\text{ad}}$.
Furthermore, we denote by $\mathcal{F}_k$ the sigma-algebra generated by
all random variables drawn up to iteration $k-1$ included,
and denote the conditional expectation on $\mathcal{F}_k$ by $\EE_k$.

\begin{theorem}\label{thm:error}
    Consider the iterates generated by \eqref{eq:gradientdescentMLMC}
    and assume that at each step the maximum level of refinement $L$
    and the associated MLMC hierarchy $\left\{M_{\ell}\right\}_{\ell=1}^L$ are chosen so that
    \begin{equation}
        \label{eq:conditionmeansquareerror}
        \EE_k\left[\|\nabla \mathcal{J}(u_k)-\nabla \mathcal{J}^{\mathrm{ML}}(u_k)\|_U^2 \right]\leq \eta \|\mathcal{R}_{\tau}(u_k)\|_U^2,
    \end{equation}
    for $\eta\in(0,\infty)$.
    Then, for every $\tau\in (0,\frac{c}{2L^2})$ and $\eta$ sufficiently
    small so that $4\sqrt{\eta}+8\eta<\frac{c^2}{2L^2}$, it holds
    \begin{align}\label{eq:meansquareaccuracy}
        \EE\left[\|u^\star-u_k\|_U^2\right]\leq \rho^k\|u^\star-u^0\|_U^2,
    \end{align}
    for some $\rho\in (0,1)$.
\end{theorem}

\begin{proof}
    Observe that
    \begin{equation*}
        \begin{aligned}
            \|u^\star-u_{k+1}\|_U^2 &=\|\Pi_{U_{\text{ad}}}\left(u^\star-\tau\nabla\mathcal{J}(u^\star)\right) - \Pi_{U_{\text{ad}}}\left(u_k-\tau\nabla\mathcal{J}^{\textrm{ML}}(u_k)\right)\|^2_U\\
            &\leq \|u^\star-\tau\nabla\mathcal{J}(u^\star)-u_k+\tau\nabla\mathcal{J}^{\textrm{ML}}(u_k)\|_U^2\\
            &\leq \|u^\star-u_k\|_U^2 -2\tau\langle u^\star-u_k,\nabla \mathcal{J}(u^\star)-\nabla \mathcal{J}^{\textrm{ML}}(u_k)\rangle_U+\tau^2\|\nabla \mathcal{J}(u^\star)-\nabla \mathcal{J}^{\textrm{ML}}(u_k)\|_U^2.
        \end{aligned}
    \end{equation*}
    Taking the conditional expectation $\EE_k$ yields
    \begin{equation*}
        \begin{aligned}
            \EE_k\left[\|u^\star-u_{k+1}\|_U^2\right] &=\|u^\star-u_k\|_U^2\\
            &-2\tau\langle u^\star-u_k,\EE_k\left[\nabla \mathcal{J}(u^\star)-\nabla \mathcal{J}^{\textrm{ML}}(u_k)\right]\rangle_U\\
            &+\tau^2 \EE_k\left[\|\nabla \mathcal{J}(u^\star)-\nabla \mathcal{J}^{\textrm{ML}}(u_k)\|_U^2\right],
        \end{aligned}
    \end{equation*}
    where we used that $\EE_k\left[\|u^\star-u_k\|_U^2\right]=\|u^\star-u_k\|_U^2$ since $u_k$ is fully determined by $\mathcal{F}_k$.
    Concerning the second term, we get with~\eqref{eq:strongcoercivity} that
    \begin{equation*}
        \begin{aligned}
            \langle u^\star-u_k,\EE_k\left[\nabla \mathcal{J}(u^\star)-\nabla \mathcal{J}^{\textrm{ML}}(u_k)\right]\rangle_U&=
            \langle u^\star-u_k,\EE_k\left[\nabla \mathcal{J}(u^\star)-\nabla \mathcal{J}(u_k)\right]\rangle_U\\
            &+\langle u^\star-u_k,\EE_k\left[\nabla \mathcal{J}(u_k)-\nabla \mathcal{J}^{\textrm{ML}}(u_k)\right]\rangle_U\\
            &\geq c\|u^\star-u_k\|_U^2-\|u-u_k\|_U\|\EE_k\left[\nabla \mathcal{J}(u_k)-\nabla \mathcal{J}^{\textrm{ML}}(u_k)\right]\|_U.
        \end{aligned}
    \end{equation*}
    For the third term, we obtain with~\eqref{eq:lipschitz_gradients}
    \begin{equation*}
        \begin{aligned}
            \EE_k\left[\|\nabla \mathcal{J}(u^\star)-\nabla \mathcal{J}^{\textrm{ML}}(u_k)\|_U^2\right]&=
            \EE_k\left[\|\nabla \mathcal{J}(u^\star)-\nabla \mathcal{J}(u_k)+\nabla \mathcal{J}(u_k)-\nabla \mathcal{J}^{\textrm{ML}}(u_k)\|_U^2\right]\\
            & \leq 2 \|\nabla \mathcal{J}(u^\star)-\nabla\mathcal{J}(u_k)\|_U^2
            +2\EE_k\left[\|\nabla \mathcal{J}(u_k)-\nabla\mathcal J^{\mathrm{ML}}(u_k)\|_U^2\right]\\
            & \leq 2 L^2 \|u^\star-u_k\|_U^2 + 2\eta \|\mathcal{R}_{\tau}(u_k)\|_U^2.
        \end{aligned}
    \end{equation*}
    Observe now that
    \begin{align*}
        \|\EE_k\left[\nabla \mathcal{J}(u_k)-\mathcal{J}^{\mathrm{ML}}(u_k) \right]\|_U\leq \left( \EE_k\left[\|\nabla \mathcal{J}(u_k)-\nabla\mathcal{J} ^{\mathrm{ML}}(u_k)\|_U^2\right]\right)^{\frac{1}{2}}\leq \sqrt{\eta}\|\mathcal{R}_{\tau}(u_k)\|_U,
    \end{align*}
    so that by collecting all pieces we derive the estimate
    \begin{align*}
        \EE_k\left[ \|u^\star-u_{k+1}\|_U^2\right]\leq \left(1-2\tau c +2\sqrt{\eta}(1+\sqrt{1-c\tau}) + 2\tau^2L^2 +2\eta\left(1+\sqrt{1-c\tau}\right)^2\right) \|u^\star-u_{k}\|_U^2.
    \end{align*}
    Note that
    \begin{align*}
        \rho \coloneqq 1-2\tau c + 2\sqrt{\eta}(1+\sqrt{1-c\tau}) + 2\tau^2 L^2 +2\eta\left(1+\sqrt{1-c\tau}\right)^2\leq 1-2\tau c+ 2\tau^2L^2 + 4\sqrt{\eta}+8\eta,
    \end{align*}
    and, for every $\tau\in (0,\frac{c}{2L^2})$ and $\eta$ sufficiently small
    so that $4\sqrt{n}+8\eta<\frac{c^2}{2L^2}$, it holds that $\rho<1$.
    Finally, by using the tower property the claim follows.
\end{proof}

Condition \eqref{eq:conditionmeansquareerror} consists in
a bound on the mean square error (MSE) of the MLMC gradient estimation.
The next proposition bounds this term by the MSEs
of the two auxiliary random variables $X$ and $Y$ defined in \eqref{eq:definition_X_Y}.
Recall that, since $U_{\mathrm{ad}}$ is bounded and the family
$\set{A(\bsxi)}_{\bsxi\in\boxN}$ has uniformly bounded inverses,
both $y$ and $q$ are in $L^\infty(\boxN;V)$,
and consequently also $y_{\ell}$ and $q_{\ell}$
due to the stability of the FE discretization.

\begin{proposition}\label{proposition:bound_mean_square_error}
    For every $u\in U_{\mathrm{ad}}$, the MSE of
    the MLMC gradient estimator satisfies
    \begin{equation}
        \label{eq:mse-gradient-estimate-1}
        \begin{aligned}
            \EE_k \big[ \| \nabla \mathcal{J}(u)-\nabla \mathcal{J}^{\mathrm{ML}}(u)\|_U^2  \big]
            & \leq 2 \Bigl(C \; \EE_k \bigl[\left|\EE\left[Y\right]-\EE^{\mathrm{ML}}\left[Y_L\right]\right|^2 \bigr] \\
            &\hspace{-0.7cm}
            + C D_L \, \EE_k \bigl[\left|\EE[Y_L]-\EE^{\mathrm{ML}}[Y_L]\right|^2 \bigr]
            + \EE_k \bigl[ \left\| \EE\left[X\right]-\EE^{\mathrm{ML}}\left[X_L\right]\right\|_U^2 \bigr]\Bigr),
        \end{aligned}
    \end{equation}
    where $C \coloneqq \frac{\|\mathbb{E}[X]\|_U^2}{\EE\left[Y\right]^2}$
    and $D_L \coloneqq \frac{(\EE[Y]-1)^2}{(\EE[Y_L]-1)^2}$.
\end{proposition}
\begin{proof}
    Inserting~\eqref{eq:exact-gradient},~\eqref{eq:gradient_estimator},
    expanding the fractions and adding zero with $\pm\EE\left[X\right]\EE\left[Y\right]$, we have
    \begin{equation*}
        \begin{aligned}
            \EE_k \big[ \| \nabla \mathcal{J}(u)-\nabla \mathcal{J}^{\mathrm{ML}}(u) \|_U^2 \big]
            &=\EE_k \Biggl[
                \left\|
                    \frac{\EE\left[X\right]\EE_{\geq 1}^{\textrm{ML}}\left[Y_L\right]-\EE^{\textrm{ML}}\left[X_L\right]\EE\left[Y\right]\pm\EE\left[X\right]\EE\left[Y\right] }
                    {\EE\left[Y\right]\EE_{\geq 1}^{\textrm{ML}}\left[Y_L\right]}
                \right\|_U^2
                \Biggr] \\
            &\leq 2 \left(  C \, \EE_k \bigl[\abs{\EE\left[Y\right]-\EE_{\geq 1}^{\textrm{ML}}\left[Y_L\right]}^2 \bigr]
                        + \EE_k \bigl[ \left\| \EE\left[X\right]-\EE^{\textrm{ML}}\left[X_L\right]\right\|_U^2\bigr] \right),
        \end{aligned}
    \end{equation*}
    where we used that $\EE_{\geq 1}^{\textrm{ML}} \left[Y_L\right] \geq 1$.
    Now, let $A\in\mathcal{F}$ be the event $A \coloneqq \left\{ \mathbb{E}^{\mathrm{ML}}\left[Y_L\right]\geq 1\right\}$.
    Using the law of total expectation, we obtain
    \begin{equation*}
        \begin{aligned}
            \EE_k\left[\left|\EE[Y]-\EE_{\geq 1}^{\textrm{ML}}[Y_L]\right|^2\right]
            &=\PP(A)\EE_k\left[\left|\EE[Y]-\EE_{\geq 1}^{\textrm{ML}}[Y_L]\right|^2\Bigl| A\right]
            +\PP(A^c)\EE_k\left[\left|\EE[Y]-\EE_{\geq 1}^{\textrm{ML}}[Y_L]\right|^2\Bigl| A^c\right]\\
            &\leq \EE_k\left[\left|\EE[Y]-\EE^{\mathrm{ML}}[Y_L]\right|^2\right]+\PP(A^c)(\EE[Y]-1)^2,
        \end{aligned}
    \end{equation*}
    where in the last step we used that $|\EE[Y]-\EE^{\mathrm{ML}}[Y_L]|^2\geq 0$.
    Expressing the probability of $A^c$ as
    \begin{equation*}
        \PP(A^c)=\PP(\EE^\textrm{ML}[Y_L]< 1)=\PP(\EE^\textrm{ML}[Y_L]-\EE[Y_L]< 1-\EE[Y_L]) ,
    \end{equation*}
    we can bound it using Chebyshev's inequality and $\EE[Y_L] > 1$ by
    \begin{equation*}
        \PP(A^c) \leq \PP\left(\left|\EE[Y_L]-\EE^{\textrm{ML}}[Y_L]\right|\geq \EE[Y_L]-1\right)
        \leq
        \frac{\EE_k\left[|\EE[Y_L]-\EE^{\textrm{ML}}[Y_L]|^2\right]}{(\EE[Y_L]-1)^2} \,.
    \end{equation*}
    Collecting the estimates leads to the claim.
\end{proof}

Next, we wish to provide a complexity estimate of the adaptive
stochastic gradient algorithm with MLMC gradient estimation~\eqref{eq:gradientdescentMLMC}.
To this end, we present the following lemma which will be relevant
to bound the MSE of the MLMC gradient estimator.

\begin{lemma}
    \label{lemma:auxiliary}
    Assume that for every $\ell\in \mathbb{N}$ there exists a null sequence $\left\{ h_{\ell}\right\}_{\ell\in\mathbb{N}}$ such that for every $u\in U_{\text{ad}}$ the corresponding
    FE state and adjoint solutions satisfy
    \begin{equation}
        \label{eq:strongerror}
        \|y_{\ell}-y_{\ell-1}\|^2_{L^2(\boxN;H)}\leq C_y h_{\ell}^\beta,
        \quad \text{and} \quad
        \|B^\ast q_{\ell}-B^\ast q_{\ell-1}\|^2_{L^2(\boxN;U)}\leq C_q h_{\ell}^\beta,
    \end{equation}
    for some exponent $\beta > 0$ independent of $u$ and constants $C_y > 0$, $C_p > 0$.
    Then, for every $\theta \in (0, \infty)$, there exist constants $C_Y(\theta) > 0$
    and $C_X(\theta) > 0$ such that for every $u\in U_{\text{ad}}$
    \begin{align}
        \label{eq:decayDenom}
        \|Y_{\ell}-Y_{\ell-1}\|^2_{L^2(\boxN;\mathbb{R})}
        &\leq C_Y(\theta) h_{\ell}^\beta \,, \\
        \label{eq:decayNum}
        \|X_{\ell}-X_{\ell-1}\|^2_{L^2(\boxN;U)}
        &\leq C_X(\theta) h_{\ell}^\beta \,.
    \end{align}
\end{lemma}

\begin{proof}
    Let $L_D$ be the Lipschitz constant of the exponential function over the ball of radius
    $R\coloneqq\theta\max\{\sup_{u\in U_{\mathrm{ad}}} \|\Phi(S[\cdot](u))\|_{L^{\infty}(\boxN;\mathbb{R})}$,
    $\sup_{u\in U_{\mathrm{ad}}} \|\Phi(S_{\ell}[\cdot](u))\|_{L^{\infty}(\boxN;\mathbb{R})}\}$,
    and $L_\Phi$ be the Lipschitz constant of $\Phi$ over the ball of radius
    $R_{\Phi}\coloneqq\sup_{u\in U_{\mathrm{ad}}} \|S_{\ell}[\cdot](u)\|_{L^{\infty}(\boxN;H)}$.
    Then, using \eqref{eq:strongerror} and the Lipschitz continuity of $\exp$ and $\Phi$,
    we obtain
    \begin{align*}
        \|\exp(\theta\Phi(y_{\ell})) - \exp(\theta\Phi(y_{\ell-1}))\|^2_{L^2(\boxN;\mathbb{R})}
        \leq (L_D\theta)^2 \|\Phi(y_{\ell}) - \Phi(y_{\ell-1}))\|^2_{L^2(\boxN;\mathbb{R})}
        \leq (L_D\theta L_\Phi C_y)^2 h_\ell^\beta,
    \end{align*}
    which proves \eqref{eq:decayDenom} with $C_Y(\theta) \coloneqq (L_D\theta L_\Phi C_y)^2$.
    Concerning \eqref{eq:decayNum}, we observe that
    \begin{align*}
        \|X_{\ell}-X_{\ell-1}\|^2_{L^2(\boxN;\mathbb{R})}
        &=
        \|\exp(\theta\Phi(y_{\ell}))B^\ast q_{\ell} - \exp(\theta\Phi(y_{\ell-1}))B^\ast q_{\ell-1}\|^2_{L^2(\boxN;U)}\\
        &\leq
        2 \Bigl(
        \|\exp(\theta\Phi(y_{\ell}))B^\ast q_{\ell} - \exp(\theta\Phi(y_{\ell-1}))B^\ast q_{\ell}\|_{L^2(\boxN;U)}^{2}\\
        &\quad\quad + \|\exp(\theta\Phi(y_{\ell-1}))B^\ast q_{\ell} -\exp(\theta\Phi(y_{\ell-1}))B^\ast q_{\ell-1}\|_{L^2(\boxN;U)}^{2}
        \Bigr) \\
        &\leq
        2 \Bigl(
        \|B^\ast q\|_{L^\infty(\boxN;U)}^{2}\|\exp(\theta\Phi(y_{\ell})) - \exp(\theta\Phi(y_{\ell-1}))\|^{2}_{L^2(\boxN;\mathbb{R})} \\
        &\quad\quad + \|\exp(\theta\Phi(y_{\ell-1}))\|^{2}_{L^\infty(\boxN;\mathbb{R})}\| B^\ast q_{\ell} -B^\ast q_{\ell-1}\|_{L^2(\boxN;U)}^{2}
        \Bigr) \,.
    \end{align*}
    Claim \eqref{eq:decayNum} then follows from the already proven \eqref{eq:decayDenom},
    assumption \eqref{eq:strongerror}, as well as the uniform
    boundedness of $q_\ell$ and $y_{\ell}$, together with the continuity of $B^\ast$ and $\Phi$.
\end{proof}

\begin{theorem}
    \label{thm:complexity_gradient_estimation}
    In addition to the assumptions of Lemma~\ref{lemma:auxiliary},
    assume that there exists an exponent $\alpha > 0$,
    and two constants $\widetilde{C}_Y(\theta) > 0$ and $\widetilde{C}_X(\theta) > 0$,
    such that for every $\theta \in (0, \infty)$,
    \begin{align}
        \label{eq:weakerror}
        \left|\mathbb{E}\left[Y-Y_\ell\right]\right|
        &\leq \widetilde{C}_Y(\theta) h_{\ell}^\alpha,\\
        \label{eq:weakerror2}
        \|\mathbb{E}\left[X-X_{\ell}\right]\|_U
        &\leq \widetilde{C}_X(\theta) h_{\ell}^\alpha,
    \end{align}
    for every $u\in U_{\mathrm{ad}}$.
    Furthermore, suppose that the cost $\mathcal{C}_\ell$ (number of floating points operations or computing time)
    to evaluate the solution operator $S_{\ell}$ and its adjoint on a level $\ell$,
    satisfies
    \begin{equation}
        \label{eq:assumption_gamma}
        \mathcal{C}_{\ell}\leq C_{\gamma} h_{\ell}^{-\gamma},
    \end{equation} for some constant $C_{\gamma} >0$ and exponent $\gamma > 0$.
    Then, for every $u\in U_{\mathrm{ad}}$
    and for any sufficiently small $\varepsilon >0$,
    there exist an integer $L$ and a MLMC hierarchy $\left\{M_{\ell}\right\}_{\ell=1}^L$,
    such that the $\varepsilon^2$-accuracy of the MLMC gradient estimator,
    \begin{align*}
        \EE_k\left[\|\nabla \mathcal{J}(u)-\nabla \mathcal{J}^{\mathrm{ML}}(u)\|_U^2 \right]\leq \varepsilon^2
    \end{align*}
    can be achieved at an asymptotic computational cost of order
    \begin{equation}
        \label{eq:cost_asymptotic}
        \mathcal{C}(\varepsilon)
        \lesssim_{\theta}
        \begin{cases}
            \varepsilon^{-2},\quad &\beta > \gamma,\\
            \varepsilon^{-2}|\log(\varepsilon)|^2,\quad &\beta= \gamma,\\
            \varepsilon^{-2-\frac{\gamma-\beta}{\alpha}},\quad &\beta< \gamma.
        \end{cases}
    \end{equation}
\end{theorem}
\begin{proof}
    Starting from Proposition~\ref{proposition:bound_mean_square_error},
    the MSE of the gradient estimator can be further bounded via the bias-variance decomposition,
    \begin{align}
        \label{eq:num-sam-decompostion-gradient}
        \EE_k \left[ \bnorm{\nabla \mathcal{J}(u)-\nabla \mathcal{J}^{\mathrm{ML}}(u)}_U^2 \right]
        \leq 2
        \Big(
        &\underbrace{\|\EE[X-X_L]\|_U^2 + C \, \babs{\EE[Y_L - Y]}^2}_{\eqqcolon \mathrm{err}^{\mathrm{num}}} \\
        &\hspace{-3cm} + \underbrace{\mathbb{E}_k \left[ \|\mathbb{E}\left[X_{L}\right]-\mathbb{E}^{\mathrm{ML}}\left[X_{L}\right]\|_U^2\right] + C(D_L+1) \, \mathbb{E}_k \left[ \babs{\mathbb{E}\left[Y_{L}\right]-\mathbb{E}^{\mathrm{ML}}\left[Y_{L}\right]}^2\right] }_{\eqqcolon \mathrm{err}^{\mathrm{sam}}} \notag.
        \Big)
    \end{align}
    On the one hand, using the assumptions \eqref{eq:weakerror}-\eqref{eq:weakerror2}, we obtain
    \begin{equation}
        \label{eq:bound_errnum}
        \mathrm{err}^{\mathrm{num}}\leq \widetilde{C}^2_X(\theta)h_L^{2\alpha} + C \widetilde{C}^2_Y(\theta)h_L^{2\alpha},
    \end{equation}
    where $C$ is the constant introduced in Proposition~\ref{proposition:bound_mean_square_error}.
    On the other hand, concerning the variance contribution, we first remark that
    \begin{align*}
        D_L=\frac{(\EE[Y]-1)^2}{(\EE[Y_L]-1)^2}=\frac{(\EE[Y]-1)^2}{((\EE[Y]-1) + \EE[Y_L-Y] )^2}.
    \end{align*}
    Due to \eqref{eq:weakerror}, there exists a $L_0\in\mathbb{N}$ such that $|\EE[Y-Y_L]|\leq \frac{\EE[Y]-1}{2}$
    for every $L\geq L_0$, leading to the bound $D_L\leq 4$ for any $L\geq L_0$.
    Therefore, for any $L \geq L_0$, Lemma~\ref{lemma:auxiliary} implies that
    \begin{equation}
        \label{eq:bound_errsam}
        \mathrm{err}^{\mathrm{sam}}\leq \left( C_X(\theta)+5CC_Y(\theta) \right)\sum_{\ell=1}^L\frac{h_{\ell}^\beta}{M_{\ell}} \,.
    \end{equation}
    Inserting \eqref{eq:bound_errnum} and \eqref{eq:bound_errsam} into
    \eqref{eq:num-sam-decompostion-gradient} leads to the estimate
    \begin{align*}
        \EE_k \big[ \| \nabla \mathcal{J}(u)-\nabla \mathcal{J}^{\mathrm{ML}}(u)\|_U^2 \big]
        \leq 2 {(C \widetilde{C}_Y^2(\theta) + \widetilde{C}_X^2(\theta))h^{2\alpha}_{L}}
        + 2 {(5 C C_Y(\theta)+C_X(\theta))\sum_{\ell=1}^{L} \frac{h_{\ell}^{\beta}}{M_{\ell}}}
    \end{align*}
    provided that $L\geq L_0$.
    The $\varepsilon^2$-accuracy can be achieved
    by choosing a suitable integer $L$ and the MLMC hierarchy $\left\{M_{\ell}\right\}_{\ell=1}^L$
    so that both upper bounds of $\mathrm{err}^{\mathrm{num}}$ and $\mathrm{err}^{\mathrm{sam}}$
    are smaller than $\frac{\varepsilon^2}{2}$.
    In particular, note that $\mathrm{err}^{\mathrm{num}} \leq \frac{\varepsilon^2}{2}$
    requires $L\rightarrow \infty$ as $\varepsilon\rightarrow 0$.
    Hence, for any sufficiently small $\varepsilon > 0$,
    the corresponding $L=L(\varepsilon)$ satisfies $L\geq L_0$
    so that the bound \eqref{eq:bound_errsam} is valid.
    The asymptotic complexity result \eqref{eq:cost_asymptotic}
    follows then from MLMC theory, see, e.g.~\cite{cliffe2011multilevel,giles2008multilevel}.
\end{proof}

Theorem~\ref{thm:complexity_gradient_estimation} quantifies the
(asymptotic) computational cost required to achieve an MSE
of the MLMC gradient estimation below a threshold $\varepsilon^2$.
Since the cost of each optimization step is largely dominated by
the gradient computation, the total cost of the algorithm
can be expressed as $\mathcal{C}^{\text{opt}}=\sum_{k=1}^K \mathcal{C}(\varepsilon_k)$,
where $K$ denotes the total number of iterations and $\varepsilon_k^2$
is the gradient estimator accuracy at the $k$-th step.
However, the accuracies $\{\varepsilon_k\}_{k=1}^K$ are themselves random variables,
and computing the expected computational cost $\mathbb{E}\LQ \mathcal{C}^{\text{opt}}\RQ $
requires to estimate the moments of inverse powers of $\{\varepsilon_k\}_{k=1}^K$
(e.g., $\mathbb{E}\LQ \varepsilon_k^{-2}\RQ$). In the following corollary,
we simplify the analysis by considering an idealized algorithm that enforces
at each iteration an expected accuracy.
This result underestimates the true complexity.
Nevertheless, we will show in the numerical experiments of
Section~\ref{sec:numerical-experiments} that this corollary
describes the asymptotic complexity well in practice.

\begin{corollary}
    \label{cor:complexity-result}
    Consider an idealized algorithm that at each iteration imposes the idealized mean-square accuracy condition,
    \begin{equation}
        \label{eq:conditionmeansquareerror_idealized}
        \EE_k\left[\|\nabla \mathcal{J}(u_k)-\nabla \mathcal{J}^{\mathrm{ML}}(u_k)\|_U^2 \right]\leq \eta \mathbb{E}\LQ \|\mathcal{R}_{\tau}(u_k)\|_U^2\RQ.
    \end{equation}
    Assuming the optimization algorithm satisfies \eqref{eq:meansquareaccuracy}, the overall cost of the resulting optimization algorithm to achieve an expected accuracy $\varepsilon$ satisfies
    \begin{equation*}
        \mathcal{C}^{\text{opt}}\lesssim_{\theta}
        \begin{cases}
            \varepsilon^{-2}\quad & \beta>\gamma,\\
            \varepsilon^{-2}|\log \varepsilon|^2\quad & \beta=\gamma,\\
            \varepsilon^{-2-\frac{\gamma-\beta}{\alpha}}\quad & \beta<\gamma.
        \end{cases}
    \end{equation*}
\end{corollary}

\begin{proof}
    Due to \eqref{eq:meansquareaccuracy},
    we need to perform $K \coloneqq \bigg\lceil\left|\frac{2\log \varepsilon}{\log \rho}\right|\bigg\rceil$ iterations
    to achieve an {\em expected} accuracy $\varepsilon^2 > 0$.
    At each iteration, the cost is dominated by the gradient
    estimation so that~\eqref{eq:conditionmeansquareerror_idealized} holds.
    In expectation, the right hand side of \eqref{eq:conditionmeansquareerror} satisfies using
    \eqref{eq:boundreducedgradient},
    \begin{align*}
        \mathbb{E}\left[ \eta \|\mathcal{R}_{\tau}(u_k)\|_U^2\right]\lesssim_{\eta,\tau} \mathbb{E}\left[\|u-u_k\|_U^2\right]
        \lesssim_{\eta,\tau}
        \rho^k \eqqcolon e_k^2.
    \end{align*}
    Then using Theorem~\ref{thm:complexity_gradient_estimation},
    the expected complexity can be bounded by
    \begin{align*}
        \mathcal{C}\lesssim_{\theta}
        \begin{cases}
            \sum_{k=1}^K e^{-2}_k\quad & \beta>\gamma,\\
            \sum_{k=1}^K e^{-2}_k|\log e_k|^2\quad & \beta=\gamma,\\
            \sum_{k=1}^K e^{-2-\frac{\gamma-\beta}{\alpha}}_k\quad & \beta<\gamma
        \end{cases}\quad
        \lesssim_{\theta} \quad
        \begin{cases}
            e^{-2}_K\quad & \beta>\gamma,\\
            e^{-2}_K|\log e_k|^2\quad & \beta=\gamma,\\
            e^{-2-\frac{\gamma-\beta}{\alpha}}_K\quad & \beta<\gamma,
        \end{cases}
    \end{align*}
    where in the last step we use a geometric series.
    The claim then follows since $e_K\cong\varepsilon$.
\end{proof}

    \section{Algorithmic and Implementation Details}
\label{sec:algorithmic-and-implementation-details}

To implement the iteration scheme~\eqref{eq:gradientdescentMLMC}
and to confirm our theoretical results, we use the framework of the adaptive
MLSGD algorithm proposed in~\cite{baumgarten2025multilevel},
together with the gradient estimator~\eqref{eq:gradient_estimator}.
Crucial to achieve the complexity of the gradient estimator
stated in Theorem~\ref{thm:complexity_gradient_estimation} in practice,
is the appropriate selection of the batch size and of the
largest level \textit{on-the-fly} at each iteration.
Since these change in every step, we attach to all previously introduced
quantities an additional index $k$, i.e.,
$\set{M_{k,\ell}}_{{\ell=1}}^{L_k}$ and $L_k$ are the multilevel batch size
and the largest level ensuring that an accuracy $\varepsilon_k^2 > 0$ is achieved at iteration $k$.
Details on this are given in Subsection~\ref{subsec:adaptive-sampling-and-mesh-refinement}.

Once the optimal multilevel batch size $\set{M_{k,\ell}}_{{\ell=1}}^{L_k}$ has been determined,
we have to find an effective way to distribute the computational load across multiple processing units.
While batched SGD methods can be embarrassingly parallel in each optimization step,
a multilevel batch is more challenging to distribute.
Lower and cheap levels require a sample parallelization,
while the high and expensive levels call for spatial
domain decomposition to exploit parallel computing systems.
Subsection~\ref{subsec:batch-parallelization} details the
consequences of this for a simplified small batch example in two spatial dimensions
and on $P=16$ processing units.

The whole procedure is summarized in Algorithm~\ref{alg:mlsgd}.
Explanations of the \texttt{MLSGD} function are provided in
Subsection~\ref{subsec:adaptive-sampling-and-mesh-refinement}.
Details of the batch distribution within the \texttt{GradientEstimate}
function and the data-merging procedures within \texttt{PairwiseUpdate}
are discussed in Subsection~\ref{subsec:batch-parallelization}.

    \subsection{MLSGD with Adaptive Sampling and Mesh Refinement}
\label{subsec:adaptive-sampling-and-mesh-refinement}

We know from Theorem~\ref{thm:complexity_gradient_estimation},
in particular from equation~\eqref{eq:num-sam-decompostion-gradient}, that
the MSE in the gradient estimation
decomposes into two components: the numerical error $\mathrm{err}_k^{\mathrm{num}}$
and the sampling error $\mathrm{err}_k^{\mathrm{sam}}$.
To choose both the finest refinement level and the optimal batch size,
knowledge of the constants, which depend on the risk-aversion parameter $\theta$,
as well as of the rates $\alpha$ and $\beta$, is required
to control both error components and achieve optimal complexity.

To this end, we estimate the error components
of~\eqref{eq:num-sam-decompostion-gradient} online as the algorithm runs
using techniques discussed in detail in~\cite{baumgarten2025budgeted}.
Particularly, we estimate on each level $\ell$ the second order power sums
\begin{equation}
    \label{eq:sec-order-sum}
    v_{k, \ell}^2 \coloneqq
    \sum_{m=1}^{M_{k, \ell}}
    \big\|\Delta X_{\ell}^{(m,k)} - \EE^{\text{MC}}_{k}[\Delta X_\ell] \big\|_U^2
    + \widehat{C} \sum_{m=1}^{M_{k, \ell}} \big|\Delta Y_{\ell}^{(m,k)} - \EE^{\text{MC}}_{k}[\Delta Y_\ell] \big|^2
\end{equation}
where $\widehat{C} \coloneqq \frac{\|\EE^{\mathrm{ML}}_k[X_L]\|_U^2}{\EE^{\mathrm{ML}}_k[Y_L]^2}$
is a sample-based approximation of the constant $C$
of Proposition~\ref{proposition:bound_mean_square_error},
and $\EE^{\text{MC}}_{k}$ represents the sample average over $M_{k,\ell}$
realizations\textsuperscript{1}\footnote{
    \textsuperscript{1}In the numerical experiments,
    we neglect the constant $D_L$ used for the theoretical
    analysis due to the use of the biased estimator
    $\EE^{\textrm{ML}}_{\geq 1}[\cdot]$.
    We never encountered an instance in which
    $\EE^{\textrm{ML}}[Y_L]$ was smaller than one.
}.
Similar to~\cite[Sec.~3]{Giles_2015},
the optimal batch size $\tset{M_{k,\ell}}_{\ell=1}^{L_k}$ is determined
through minimizing the computational cost while achieving
the target sampling error $\mathrm{err}_k^{\mathrm{sam}} = \varepsilon_k^2/2$
in each iteration $k$, yielding
\begin{equation}
    \label{eq:optimal-batch-size}
    M_{k, \ell}^{\text{opt}}
    = \left\lceil
          \left(\sqrt{2} \varepsilon_{k}\right)^{-2}\!\!
          \sqrt{\frac{v_{k-1,\ell}^2}{(M_{k - 1, \ell} - 1) \, \mathcal{C}_{k - 1, \ell}}}
          \left(
              \sum_{\ell'=1}^{L_{k}}
              \sqrt{\frac{v_{k,\ell'}^2 \, \mathcal{C}_{k-1, \ell'}}{M_{k - 1, \ell'} - 1}} \,\,
          \right)
    \right\rceil \,.
\end{equation}
Here, $\mathcal{C}_{k - 1, \ell}$ is not estimated (e.g., as a function of the size of the finite element approximation space),
but it is the actual measured average cost on level $\ell$ of the previous iteration in CPU seconds.
The optimal level $L_k$ is chosen as either $L_{k-1}$ or $L_{k-1} + 1$,
depending on whether the estimated bias exceeds the prescribed tolerance.
Motivated by~\eqref{eq:num-sam-decompostion-gradient}
and~\cite[Sec.~3]{Giles_2015},
we estimate the convergence rate $\widehat{\alpha}$ experimentally
via a least-squares fit and define the squared bias estimate by
\begin{equation}
    \label{eq:squared-bias-estimate-mlmc}
    \widehat{\mathrm{err}}^{\text{num}}_k
    \coloneqq \max_{\ell = 2, \dots, L} \left( \frac{\|\EE^{\text{MC}}_k[\Delta X_\ell] \|^2 + \widehat{C} \, | \EE^{\text{MC}}_k[\Delta Y_\ell] |^2} {(((h_{\ell-1} / h_\ell)^{\widehat{\alpha}} - 1) (h_L/h_\ell)^{\widehat{\alpha}})^2} \right) \,.
\end{equation}
The level update is then given by
\begin{equation}
    \label{eq:optimal-level}
    L_k =
    \begin{cases}
        L_{k-1} + 1, & \text{ if } \widehat{\mathrm{err}}^{\text{num}}_k > \varepsilon_k^2/2, \\[1ex]
        L_{k-1}, & \text{ otherwise.}
    \end{cases}
\end{equation}
Algorithm~\ref{alg:mlsgd}, starting in the \texttt{MLSGD} function,
outlines an implementation of
the introduced estimators in the optimization process.
The illustrated algorithm aims to achieve a gradient norm
smaller than some chosen accuracy $\tilde{\varepsilon} > 0$,
using the MSE estimate
\begin{equation}
    \label{eq:mse-estimate-mlmc}
    \widehat{\mathrm{err}}^{\text{mse}}_k \coloneqq
    \widehat{\mathrm{err}}^{\text{num}}_k + \widehat{\mathrm{err}}^{\text{sam}}_k
    \quad \text{with} \quad
    \widehat{\mathrm{err}}^{\text{sam}}_k\coloneq \sum_{\ell=1}^{L_k} \frac{v_{k, \ell}^2}{M_{k, \ell}(M_{k, \ell}-1)}
\end{equation}
to monitor both error contributions.

While the \texttt{MLSGD} function controls the batch size and performs the
iteration scheme~\eqref{eq:gradientdescentMLMC},
the \texttt{GradientEstimate} function,
called for every optimization step,
dominates the cost of the algorithm as multiple PDEs have to be solved
on different discretization levels and for different input samples
in order to construct~\eqref{eq:shorthand-notation}
and~\eqref{eq:delta-notation}.
To mitigate this high cost, we outline in the
following how to distribute the PDE solves.

\begin{algorithm}
    \caption{Distributed Adaptive Multilevel Stochastic Gradient Descent (MLSGD)}
    \label{alg:mlsgd}
    \small{
        \begin{align*}
            &\texttt{function MLSGD}(u_0, \{M_{0, \ell}\}_{\ell=1}^{L_0}, \tilde{\varepsilon}) \colon \\
            &\quad
            \begin{cases}
                \texttt{while } \tnorm{\nabla \mathcal{J}^{\mathrm{ML}}(u_k)}_U > \tilde{\varepsilon} \colon \\
                \quad
                \begin{cases}
                    \text{// Determine the level $L_k$ and the optimal batch $\{M_{k,\ell}\}_{\ell=1}^{L_k}$} \\
                    \texttt{if } k \neq 0 \colon \, L_k \leftarrow \eqref{eq:optimal-level} \quad
                    \{M_{k,\ell}\}_{\ell=1}^{L_k} \leftarrow \eqref {eq:optimal-batch-size} \\[1mm]
                    \text{// Compute the gradient estimate following~\eqref{eq:gradient_estimator}} \\
                    \nabla \mathcal{J}^{\mathrm{ML}}(u_k) \leftarrow \texttt{GradientEstimate}(u_k, \{M_{k,\ell}\}_{\ell=1}^{L_k}) \\[1mm]
                    \text{// Update control considering feasible set following~\eqref{eq:gradientdescentMLMC}} \\
                    u_{k+1} \leftarrow \Pi_{U_{\text{ad}}} \left(u_{k} - \tau \nabla \mathcal{J}^{\mathrm{ML}}(u_k)\right) \\[1mm]
                    \text{// Update iteration } k \text{ and target } \varepsilon_k \text{ with } \eta \in (0, 1) \\
                    k \leftarrow k + 1 \quad \varepsilon_k \leftarrow \eta \tnorm{\nabla \mathcal{J}^{\mathrm{ML}}(u_k)}_U
                \end{cases} \\
                \texttt{return } u_{k+1}
            \end{cases} \\
            &\texttt{function GradientEstimate}(u_k, \{M_{k,\ell}\}_{\ell=1}^{L_k}) \colon \\[-1mm]
            &\quad
            \begin{cases}
                \texttt{for } \ell = 1 \texttt{ to } L \colon \\
                \quad
                \begin{cases}
                    \text{// Distribute PDE solves as outlined in Subsection~\ref{subsec:batch-parallelization}} \\
                    \texttt{run in parallel for } m = 1, \dots, M_{k,\ell} \colon \\
                    \quad
                    \begin{cases}
                        \text{// Draw } \bsxi^{(m,\ell)}\!, \text{ solve state and adjoint on } \ell \text{ and } \ell - 1 \\
                        \tset{X_{\ell}^{(m)} \!, X_{\ell-1}^{(m)} \!, \, Y_{\ell}^{(m)} \!, \, Y_{\ell-1}^{(m)}}_{m=1}^{M_{k,\ell}}
                        \leftarrow \eqref{eq:shorthand-notation} \\
                        \tset{\Delta X_{\ell}^{(m)} \!, \, \Delta Y_{\ell}^{(m)}}_{m=1}^{M_{k,\ell}}
                        \leftarrow \eqref{eq:delta-notation}
                    \end{cases} \\\\[-2mm]
                    \text{// Merge distributed data to enable control update} \\
                    \EE_{k}^{\mathrm{MC}}[\Delta X_\ell], \, \EE_{k}^{\mathrm{MC}}[\Delta Y_\ell]
                    \leftarrow \texttt{PairwiseUpdate}(\tset{\Delta X_{\ell}^{(m)}, \Delta Y_{\ell}^{(m)}}_{m=1}^{M_{k,\ell}}) \\[1mm]
                    \text{// Compute level sums~\eqref{eq:ml_estimator_numerator} and~\eqref{eq:ml_estimator_denominator}} \\
                    \EE^{\mathrm{ML}}_k[X_{\ell}] \leftarrow \EE^{\mathrm{ML}}_k[X_{\ell-1}] + \EE_k^{\mathrm{MC}}[\Delta X_{\ell}] \\
                    \EE^{\mathrm{ML}}_k[Y_{\ell}] \hspace{0.6mm} \leftarrow \EE^{\mathrm{ML}}_k[Y_{\ell-1}] + \EE_k^{\mathrm{MC}}[\Delta Y_{\ell}]
                \end{cases} \\\\[-2mm]
                \texttt{return } \EE^{\mathrm{ML}}_k[X_L] / \EE^{\mathrm{ML}}_k[Y_L] + \lambda u_k
            \end{cases} \\
            &\texttt{function PairwiseUpdate}(\tset{\Delta X_{\ell}^{(m)}, \Delta Y_{\ell}^{(m)}}_{m=1}^{M_{k,\ell}}) \colon \\
            &\quad
            \begin{cases}
                \text{// Determine communication split } s_{k,\ell} \text{ and initialize starting distance} \\
                s_{k, \ell} \leftarrow \eqref{eq:optimal-parallelization}
                \quad
                \tilde d \leftarrow \log_2(P) - 1 \\[1mm]
                \texttt{while } \tilde d \geq \log_2(P) - s_{k,\ell} \colon \\[0mm]
                \quad
                \begin{cases}
                    &\hspace{-18mm} \text{// Join estimates of process groups $A$ and $B$ at distance $d=2^{\tilde d}$} \\
                    \,\,\,\,\, M_{AB}      &\leftarrow \,\,\, M_{A} \, +_d \, M_{B} \\
                    \delta_{AB}[\Delta Y_\ell] &\leftarrow \,\,\, \EE_A^{\mathrm{MC}}[\Delta Y_\ell] \, -_d \, \EE_B^{\mathrm{MC}}[\Delta Y_\ell] \\
                    \delta_{AB}[\Delta X_\ell] &\leftarrow \,\,\, \EE_A^{\mathrm{MC}}[\Delta X_\ell] \, \ominus_d \, \EE_B^{\mathrm{MC}}[\Delta X_\ell] \\
                    \EE_{AB}^{\mathrm{MC}}[\Delta Y_\ell] &\leftarrow \,\,\, \EE_B^{\mathrm{MC}}[\Delta Y_\ell] \,\, +_d \, \tfrac{M_A}{M_{AB}} \delta_{AB}[\Delta Y_\ell] \\
                    \EE_{AB}^{\mathrm{MC}}[\Delta X_\ell] &\leftarrow \,\,\, \EE_B^{\mathrm{MC}}[\Delta X_\ell] \,\, \oplus_d \, \tfrac{M_A}{M_{AB}} \delta_{AB}[\Delta X_\ell] \\[1mm]
                    &\hspace{-18mm} \text{// Select new sets $A$ and $B$ by reducing the distance $d=2^{\tilde d}$} \\
                    \,\,\,\,\,\, \tilde d &\leftarrow \,\,\, \tilde d - 1
                \end{cases} \\
                \texttt{return } \texttt{SelectSubdomain}(\EE_{AB}^{\mathrm{MC}}[\Delta X_\ell]), \EE_{AB}^{\mathrm{MC}}[\Delta Y_\ell]  \\[1mm]
            \end{cases} \\
            &\texttt{function SelectSubdomain}(\EE_{AB}^{\mathrm{MC}}[\Delta X_\ell]) \colon \\
            &\quad
            \begin{cases}
                \texttt{for } x \in \mathcal{D} \colon \\
                \quad
                \begin{cases}
                    \text{// Checking for condition~\eqref{eq:congruent-processes} and setting the data} \\
                    \texttt{if } \pi_0(x) \in \pi_s(x) \colon
                    \hspace{2mm} \EE_{k}^{\mathrm{MC}}[\Delta X_\ell](x) \leftarrow \EE_{AB}^{\mathrm{MC}}[\Delta X_\ell](x) \\
                    \texttt{else } \colon
                    \hspace{18mm} \texttt{Error("x not found")}
                \end{cases} \\
                \texttt{return } \EE_{k}^{\mathrm{MC}}[\Delta X_\ell]
            \end{cases}
        \end{align*}
    }
\end{algorithm}

\begin{figure}[H]
    \centering
    \includegraphics[angle=90,width=0.6\textheight]{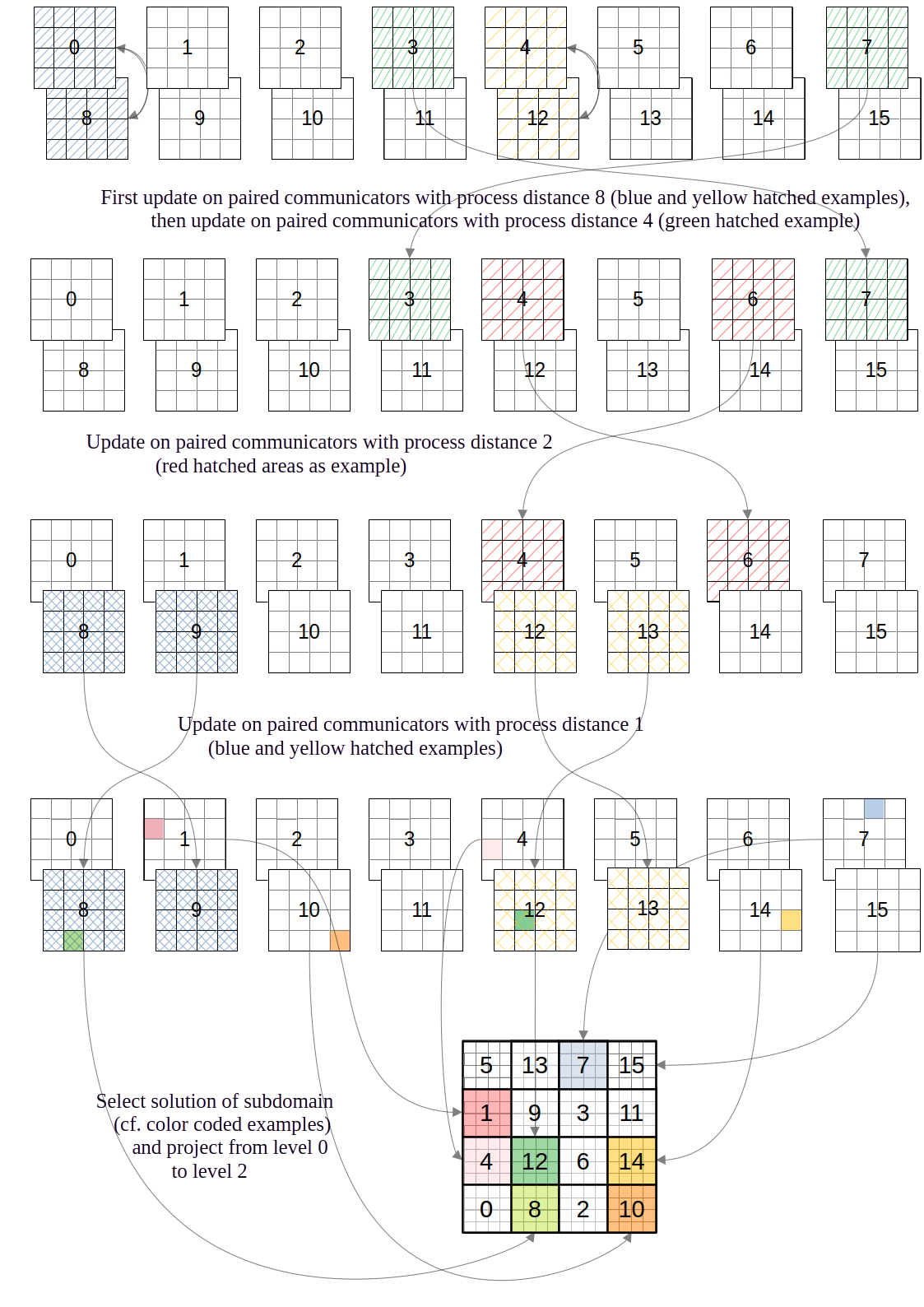}
    \caption{Update procedure for $M_{k,\ell=1} = 16$ parallel samples on level $\ell=1$.
    The result is accumulated onto one domain distributed data structure on level $\ell=3$.}
    \label{fig:update}
    \vspace{0.8cm}
    \vfill
    \includegraphics[angle=90,width=0.6\textheight]{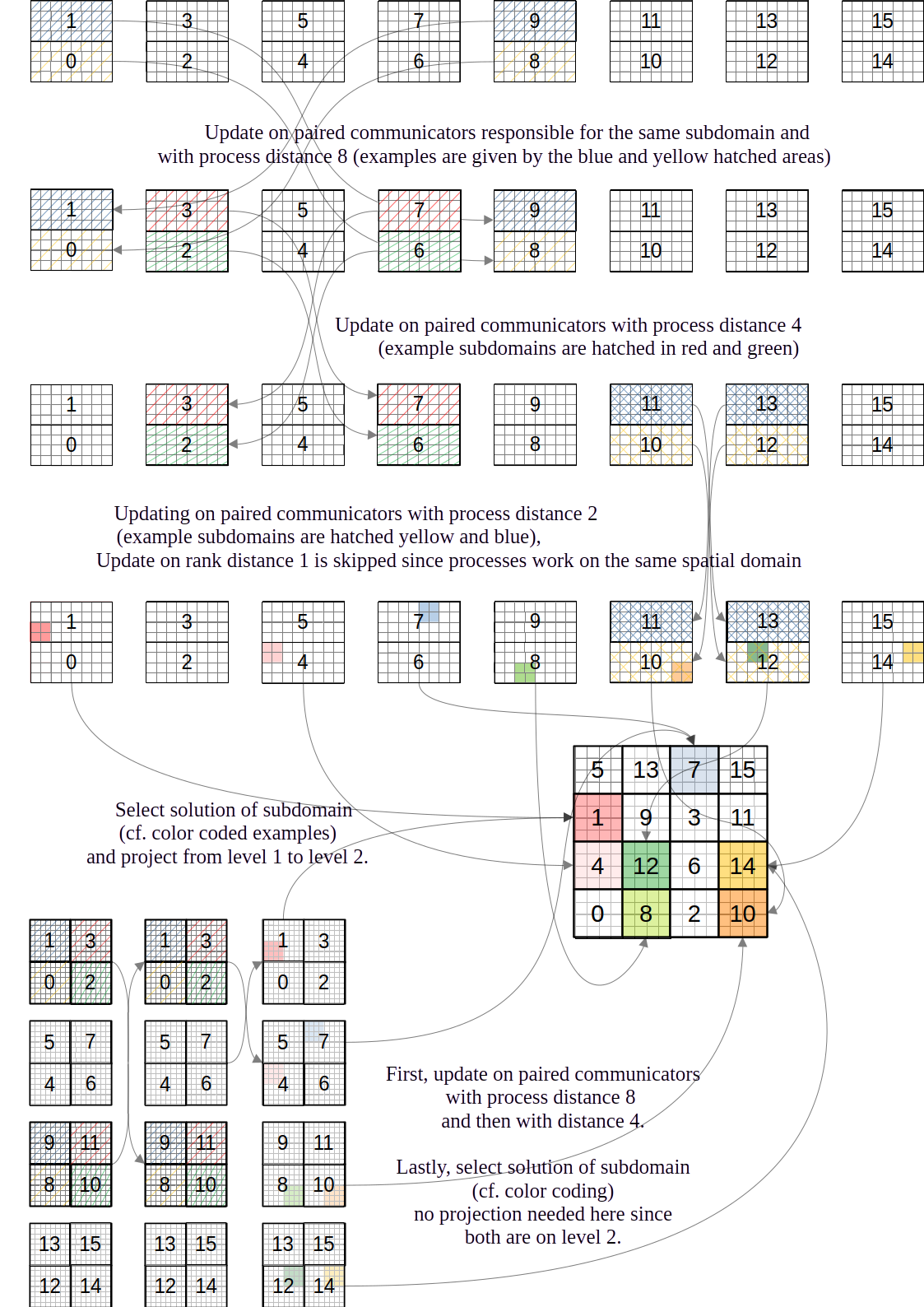}
    \caption{
        Update procedure for $M_{k,\ell=2} = 8$
        and $M_{k,\ell=3} = 4$ parallel samples.
        The result is accumulated onto one domain distributed data structure on level $\ell=3$.
    }
    \label{fig:update2}
\end{figure}
    \subsection{Batch Parallelization}
\label{subsec:batch-parallelization}

We now describe how to maximize computational efficiency
by balancing sample-level parallelism with subdomain decompositions.
Suppose the task at optimization step $k$ is to compute
\begin{equation}
    \label{eq:example-batch}
    M_{k,\ell=1} = 16, \quad M_{k,\ell=2} = 8, \quad M_{k,\ell=3} = 4
\end{equation}
samples on $P = 16$ parallel processes on their respective processing units.
The goal is to accumulate the statistics of all
samples in a parallel pass through the batch data
using the pairwise update formulas, merging successively two statistical results,
as introduced in~\cite{chan1982updating, pebay2016numerically, welford1962note}
onto a single representation on level $L=3$.
In particular, we want to be able to represent
$\EE^{\mathrm{ML}}[X_L]$ in a domain distributed FE basis
and to have $\EE^{\mathrm{ML}}[Y_L]$ available at all processes
to compute the final gradient estimator according to~\eqref{eq:gradient_estimator}.

To illustrate this, we refer to the Figures~\ref{fig:update} and~\ref{fig:update2},
where the FE mesh, on which we want to store $\EE^{\mathrm{ML}}[X_L]$,
is shown as the large quadratic mesh with $16$ subdomains, each with an assigned process numbered from $0$ to $15$.
At the hand of these figures, we describe the hierarchical communication
and update process (indicated by the arrows) and the three key ingredients
enabling the efficient accumulation of the batch data:
\begin{enumerate}
    \item[a)] a \textit{communication split index} $s_{k,\ell} \in \mathbb{N}_0$
    as in~\cite{baumgarten2025budgeted, baumgarten2024fully}.
    \item[b)] a \textit{pairwise update scheme} following~\cite{chan1982updating}.
    \item[c)] a \textit{suitable domain decomposition} to avoid unnecessary communication.
\end{enumerate}

\textit{a) Communication Split Index.}
In order to find the optimal distribution strategy for
the samples on each level $\ell$ and in each step $k$,
we have to find the optimal number of communication splits $s_{k,\ell} \in \mathbb{N}_0$.
That is, we determine by the formula
\begin{equation}
    \label{eq:optimal-parallelization}
    s_{k,\ell} = \lceil \log_2 (\min \set{P, M_{k,\ell}, M_{\ell}^{\max}}) \rceil
\end{equation}
how many times the MPI world communicator is split
into two groups of size $P_{s_{k,\ell}}$.
Here, $M_{\ell}^{\max}$ is the maximal number of samples we can compute at the same time on level $\ell$,
due to memory constraints, though we assume in this section,
that $M_{\ell}^{\max}$ is sufficiently large to not affect the optimal distribution strategy.
The index $s_{k,\ell}=0$ represents the MPI world communicator with size $P_0 = P$,
$s_{k,\ell}=1$ divides the processes into two groups of size $P_1 = P / 2$,
and $s_{k,\ell} = \log_2(P)$ completely separates the processes
into individual communicators assuming $P$ is a power of $2$.

Therefore, in example~\eqref{eq:example-batch} and
by following the formula~\eqref{eq:optimal-parallelization},
the optimal way to distribute the $M_{k,1} = 16$
samples on $P = 16$ processing units
is to completely distribute samples over all processes,
and thus $s_{k,1} = 4$ leads to a communicator size of $P_4 = 1$.
This scenario is illustrated in Figure~\ref{fig:update},
where the top row represents the 16 different samples,
each computed on a different process.
Similar to the example above, the optimal distribution
on $\ell=2$ and $\ell=3$ result in $s_{k,2} = 3$ and $s_{k,3} = 2$.
Finding the communication split $s_{k,\ell}$ is
the first step in the \texttt{PairwiseUpdate} function of Algorithm~\ref{alg:mlsgd}
to determine how the data is distributed.

\smallskip

\textit{b) Pairwise Update Scheme.}
Next, we use the formulas
from~\cite{chan1982updating, pebay2016numerically, welford1962note}
to merge the sample data with a pairwise algorithm, eventually
combining the $M_{k,\ell}$ distributed samples into a common estimates $\EE_k^{\mathrm{MC}}[\Delta X_{\ell}]$ and $ \EE_k^{\mathrm{MC}}[\Delta Y_{\ell}]$.
Figure~\ref{fig:update} illustrates this procedure.
In the top row, processes are paired with others that are $d=8$ positions apart;
for example, process~0 communicates with process~8 (blue-hatched areas),
and process~4 with process~12 (yellow-hatched areas).
The transition to the second row represents the
next update with distance $d=4$ (green-hatched example).
The same procedure is then applied for distances $d=2$
(red-hatched example) and $d=1$
(blue- and yellow-hatched examples; see the middle and lower rows of Figure~\ref{fig:update}).
This results in common estimates $\EE_k^{\mathrm{MC}}[\Delta X_{\ell}]$ and $ \EE_k^{\mathrm{MC}}[\Delta Y_{\ell}]$ based on all samples of one batch.
This data corresponds to one particular level and may be domain distributed for $\EE_k^{\mathrm{MC}}[\Delta X_{\ell}]$
and copied on each process for $ \EE_k^{\mathrm{MC}}[\Delta Y_{\ell}]$.

Since $M_{k,1} \geq \dots \geq M_{k,L}$, the communication index $s_{k,\ell}$
may differ across levels due to~\eqref{eq:optimal-parallelization},
(see also Figure~\ref{fig:nodes-experiment} on the bottom left)
leading to different data distributions.
As mentioned, with $M_{k,2}=8$ as in~\eqref{eq:example-batch}
and $P=16$, we obtain $s_{k,2}=3$, so two processing units ($P_3=2$)
are combined to compute one sample (see the top row of Figure~\ref{fig:update2}).
However, data from different levels and communication splits
$s_{k,\ell}$ must ultimately be merged into a common data
structure on the largest level and the MPI world communicator.
As illustrated in the bottom row of Figure~\ref{fig:update},
this transition, from a sample-distributed structure to
the domain-distributed global grid, is achieved by assigning
and projecting the data to the appropriate subdomains of the FE mesh.
In the figure, this assignment is marked in red for process~1,
in blue for process~7, and analogously for the
remaining processes using solid color markings.

While in Figure~\ref{fig:update} the pairwise update,
selection, and projection are straightforward,
these operations are more involved when $s_{k,\ell} \neq \log_2(P)$,
as shown in Figure~\ref{fig:update2}.
The update from levels $\ell=2$ and $\ell=3$
is illustrated using the same color coding to
highlight the pairwise communication and selection steps.
The pairwise update follows a similar idea as for $s_{k,\ell}=\log_2(P)$,
but the partial domain decomposition must be preserved.
In the top row of Figure~\ref{fig:update2}, for instance,
process~0 and process~8 (yellow-hatched areas) are
paired to update the statistics on the lower subdomain,
while process~1 and process~9 (blue-hatched areas) update the upper subdomain.
On level $\ell=2$, the update for process distance $d=1$
is omitted because processes at this distance already work on the same sample.
The same occurs on level $\ell=3$ (lower left of Figure~\ref{fig:update2}),
where $s_{k,\ell}=2$ due to $M_{k,\ell=3}=4$
and~\eqref{eq:optimal-parallelization};
here the updates for $d=1$ and $d=2$ are skipped
since all sample data has already been merged.

The update procedure, implementing the formulas from~\cite{chan1982updating, pebay2016numerically, welford1962note},
is formally summarized by the \texttt{while} loop in the \texttt{PairwiseUpdate} function
of Algorithm~\ref{alg:mlsgd}.
To this end, we introduce
the following notation for algebraic operations:
Considering a pair of process sets $A, B \subset \mathcal{P}$,
$\mathcal{P}$ being the set of all processing units,
at distance $d \in \mathbb N$, we define
\begin{align*}
    +_d \colon \mathbb{R} \times \mathbb{R} \rightarrow \mathbb{R},
    \quad
    -_d \colon \mathbb{R} \times \mathbb{R} \rightarrow \mathbb{R}
    \quad
    \oplus_d \colon V_\ell \times V_\ell \rightarrow V_\ell,
    \quad \ominus_d \colon V_\ell \times V_\ell \rightarrow V_\ell \, .
\end{align*}
The operations $+_d$ and $-_d$ denote the addition and
subtraction of two floating-point numbers,
one stored on process set $A$ and the other on process set $B$,
where as the result of the addition is then stored on both.
Similarly, the operations $\oplus_d$ and $\ominus_d$ denote
component-wise addition and subtraction
of two coefficient vectors representing solutions in the FE space $V_\ell$
where again the result is stored on both sets $A$ and $B$.
Importantly, these coefficient vectors may already be stored in a domain-distributed manner,
since the operations $\oplus_d$ and $\ominus_d$ only combine data assigned to
processes at distance $d$, thereby preserving the underlying subdomain decomposition.

Using these operations, we accumulate the final statistics,
where $M_{AB}$, $\EE^{\mathrm{MC}}_{AB}[\Delta Y_\ell]$,
and $\EE^{\mathrm{MC}}_{AB}[\Delta X_\ell]$
denote the quantities obtained by combining process sets $A$ and $B$.
The merging procedure is repeated as long as the distance between process sets
exceeds the distance associated with the domain decomposition.

\smallskip

\textit{c) Suitable Domain Decomposition.}
The selection of data in the mixed sample and domain distributed
data structures can be performed without additional communication
only if the domain decomposition across processes is chosen carefully.
For example, on levels $\ell=2$ and $\ell=3$, process~7 (solid blue area)
can access the required global grid data only if the mixed
data structures assign the same subdomain to process~7.
Otherwise, the data is not locally available
and additional communication is required.

To solve the aforementioned issue, we focus the explanation on $\ell = 3$
illustrated in the lower left corner in Figure~\ref{fig:update2}
and define $\pi_{s_{k,\ell}} \colon \mathcal{D} \rightarrow \mathcal{P}_{s_{k,\ell}}$
as the map assigning to every nodal point $x$ in the domain $\mathcal{D}$
a set of processes $\mathcal{P}_{s_{k,\ell}}(x)$ of size $\abs{\mathcal{P}_{s_{k,\ell}}(x)} \equiv P_{s_{k,\ell}}$.
In the given example for $\ell=3$,
the origin $x = 0$ in the domain $\mathcal{D}$ (lower left corner)
is assigned to $\pi_2(0) = \set{0, 4, 8, 12}$.
Since $x = 0$ is also assigned in the world communicator to $\pi_0(0) = \set{0}$,
the data is locally available to process~0,
and no communication is required for the selection step.
Hence, we enforce $\pi_0(x) \subseteq \pi_{s_{k,\ell}}(x)$
for all nodal points $x \in \mathcal{D}$ by the congruence condition
\begin{equation}
\label{eq:congruent-processes}
A \equiv B \pmod{P_{s_{k,\ell}}} \quad \text{for all} \quad A, B \in \mathcal{P}_{s_{k,\ell}}(x) = \pi_{s_{k,\ell}}(x)
\end{equation}
for any communication split $s_{k,\ell}$.
In other words, the distance $d = B - A$ of processes which store data to the same point
$x \in \mathcal{D}$ has to be a multiple of the communicator sizes $P_{s_{k,\ell}}$
for all communication splits $s_{k,\ell}$.
This selection step is the final operation performed on
$\EE^{\mathrm{MC}}_{AB}[\Delta X_\ell]$ in the \texttt{PairwiseUpdate}
function and done by the \texttt{SelectSubdomain} function
checking for condition~\eqref{eq:congruent-processes}.

In conclusion, the example outlines three ingredients of the update procedure:
a) the definition of the communication split index $s_{k,\ell}$ in~\eqref{eq:optimal-parallelization},
b) the pairwise update of the sample data following~\cite{chan1982updating},
and c) the congruence condition~\eqref{eq:congruent-processes} to ensure a
communication-free selection of the data.
As a result, the update procedure only uses the operations
that are statistically necessary by~\cite{chan1982updating}.
We also remark that we made no assumptions on the shape of the domain
or the distribution algorithm except that the congruence
condition~\eqref{eq:congruent-processes} has to be satisfied.
For example, simple recursive coordinate bisection
with appropriate process assignment can be used
to achieve~\eqref{eq:congruent-processes}.
    \section{Numerical Experiments}
\label{sec:numerical-experiments}

To examine and validate the developed algorithm
and the corresponding convergence analysis,
we conduct numerical experiments on a risk-averse three-dimensional
elliptic optimal control problem.
In particular, we use $\mathcal{D}=(0,1)^3, U=H=L^2(\mathcal D)$ and $V=\mathrm{H}^1_0(\mathcal{D})$,
and consider the elliptic problem \eqref{eq:sys} with
\begin{align*}
    \langle A(\bsxi) y(\bsxi), v\rangle_{V',V} &= \int_\mathcal{D} \exp(\tilde{a}\left(\xb,\bsxi)\right) \nabla y(\bsxi, \xb) \cdot \nabla v(\xb) \,\mathrm d \xb,\\
    \langle Bu, v\rangle_{V',V}&=\int_\mathcal{D}  u(\xb) v(\xb) \,\mathrm d \xb.
\end{align*}
Motivated by the two-dimensional example~\cite[Ex.~9.37]{lord2014stochasticPDE},
the realizations of the random coefficient $\tilde{a}$ are generated
by a truncated Karhunen--Loève Expansion
\begin{align*}
    \tilde{a}(\xb,\omega) = \mu_0(\xb) + \sum_{j,k,l=1}^{5} \sqrt{\lambda_{jkl}} \, \phi_{jkl}(\xb) \, \xi_{jkl}(\omega),
    \quad \xi_{jkl}\sim \mathcal{U} \left(-4,4\right) \,\,\, \text{iid},
\end{align*}
where $\phi_{jkl}(\xb) = \cos(j\pi x_2) \, \cos(k\pi x_3) \, \cos(l\pi x_1)$
and $\lambda_{jkl}=\exp(-\pi (j+k+l) \, \rho)$
for $j,k,l\geq 1$ with correlation length $\rho=0.15$ and mean $\mu_0\equiv 0$.
Illustrations of different samples of the random field are shown in Figure~\ref{fig:samples}.
The target state in~\eqref{eq:cost}
is chosen as $g(\xb)=\sin(2\pi x_1) \sin(2\pi x_2) \sin(2\pi x_3)$.

\begin{figure}
    \centering
    \includegraphics[width=0.40\linewidth]{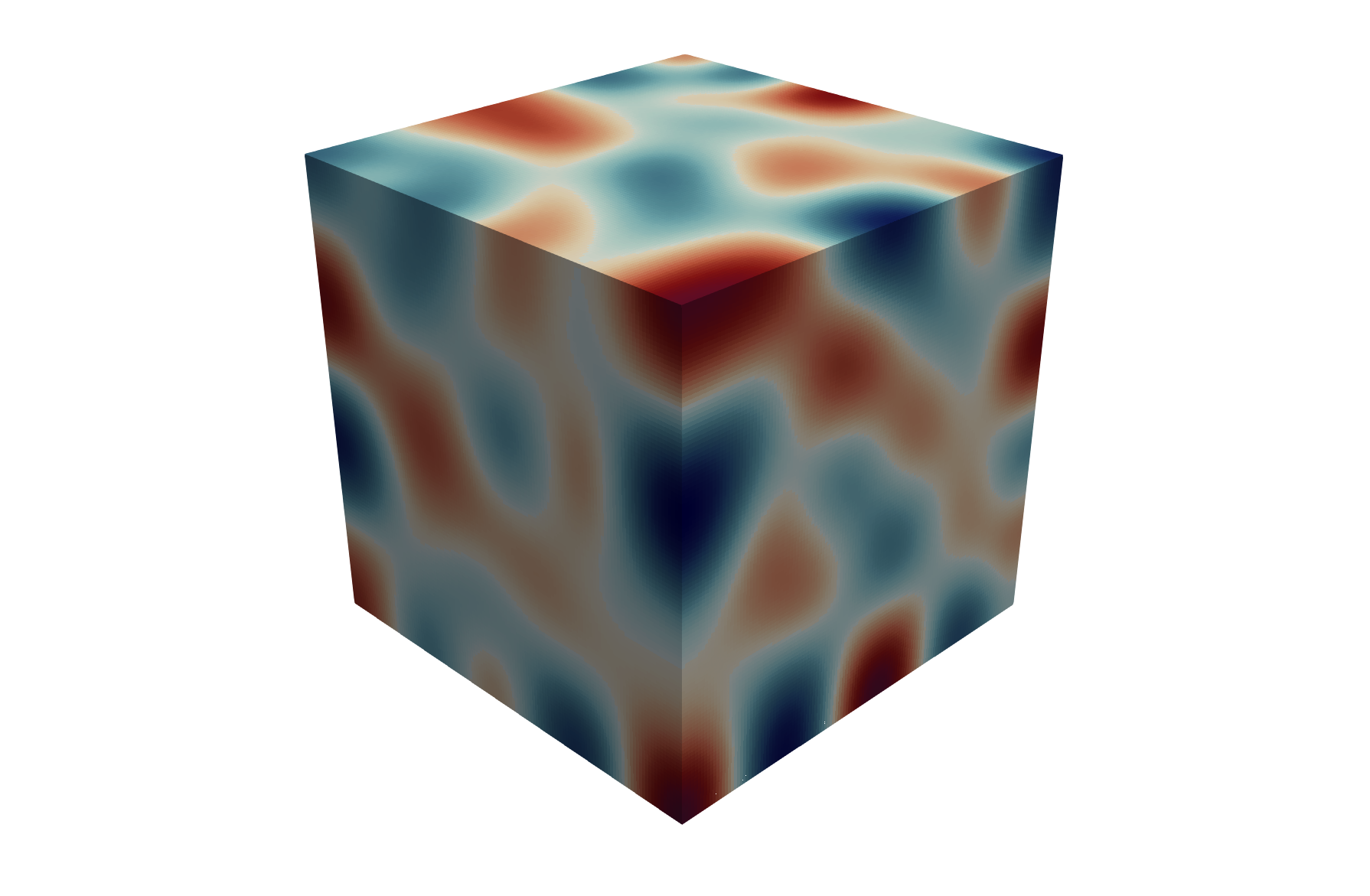}
    \hspace{-1.3cm}
    \includegraphics[width=0.40\linewidth]{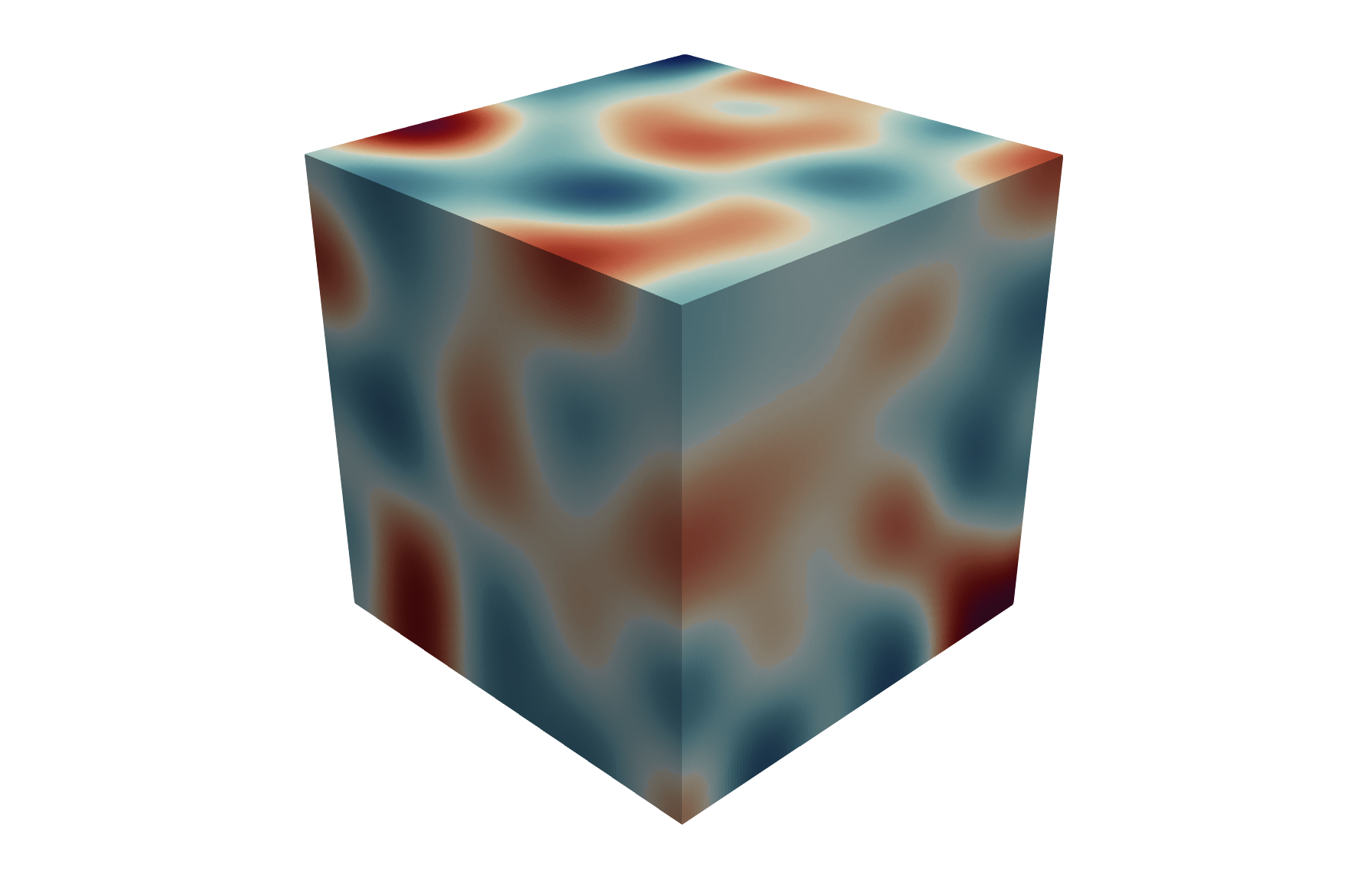}
    \caption{Two samples $\tilde{a}_1$ and $\tilde{a}_2$ of the random fields used for $A$.}
    \label{fig:samples}
\end{figure}

We use standard piecewise linear finite elements and
a multigrid preconditioned conjugate gradient method
to solve the PDEs, together with MLSGD method described in Algorithm~\ref{alg:mlsgd} and the
gradient estimator~\eqref{eq:gradient_estimator}.
We choose $\varepsilon_k=\eta \norm{\nabla \mathcal{J}^{\mathrm{ML}}(u_k)}_U$ with $\eta=0.9$ for the optimal batch size
in \eqref{eq:optimal-batch-size} to ensure, that the accuracy of Algorithm~\ref{alg:mlsgd}
is matched in~\eqref{eq:mse-estimate-mlmc}.
The target accuracy $\tilde{\varepsilon} > 0$ is chosen for all experiments such that
the total simulation time is four hours, which is achieved
through budgeting techniques as done in~\cite{baumgarten2025budgeted}.
If not stated otherwise, we use $P = 64$ CPUs to run the algorithm.

For the step size in Algorithm~\ref{alg:mlsgd} we use the adaptive procedure proposed
in~\cite{koehne2024adaptivestepsizespreconditioned}
\begin{equation}
    \label{eq:adaptive-step-size}
    \tau_{k} =
    \frac{\|{ \nabla \mathcal{J}^{\mathrm{ML}}(u_k)}\|_{U}^2 - \widehat{\mathrm{err}}^{\text{sam}}_k}
    {\widehat{L} \|{ \nabla \mathcal{J}^{\mathrm{ML}}(u_k)}\|_{U}^2}
    \quad \text{with} \quad
    \widehat{L} =
    \frac{\|{ \nabla \mathcal{J}^{\mathrm{ML}}(u_k)- \nabla \mathcal{J}^{\mathrm{ML}}(u_{k-1})}\|_{U}}
    {\|{t_{k-1} \nabla \mathcal{J}^{\mathrm{ML}}(u_{k-1})}\|_{U}},
\end{equation}
which is based on \eqref{eq:lipschitz_gradients}, however,
note that fixed step sizes satisfying  $\tau\in (0,\frac{c}{2L^2})$,
as derived in the proof of Theorem~\ref{thm:error}, work as well.

The goals of the numerical experiments are threefold:
(i) To explore the impact of $\theta$,
particularly in high risk-aversion regimes and in demanding computational settings.
(ii) To validate the presented convergence analysis of
the MLSGD method for risk-averse optimal control problems,
to verify the assumptions~\eqref{eq:strongerror}~\eqref{eq:weakerror},~\eqref{eq:weakerror2},~\eqref{eq:assumption_gamma}
as well as the results~\eqref{eq:decayDenom} and \eqref{eq:decayNum}
by measuring the multilevel rates $\alpha, \beta, \gamma$.
(iii) To investigate the performance of the method in a high-performance computing environment
and to evaluate the proposed batch parallelization and control update
in Subsection~\ref{subsec:batch-parallelization}.

    \subsection{Experiments on risk-aversion parameter}
\label{subsec:risk-scaling}

We apply the algorithm to the optimal control problem
described in the previous paragraph while varying the risk-aversion parameter
$\theta \in \set{1.0, 8.0, 16.0, \dots, 72.0}$.
Figure~\ref{fig:theta-experiment} shows three quantities
as functions of $\theta$: the value of the objective
function~\eqref{eq:cost} on the left, the $L^2$ norm of the
gradient estimate~\eqref{eq:gradient_estimator} at the last iterate in the center,
and the total number of optimization steps completed
within four hours of computation on the right.

We observe that, over this range of risk-aversion parameters,
the objective exhibits an approximately linear dependence on $\theta$,
particularly for smaller values of $\theta$.
For larger values, the relationship becomes noisier.
This behavior is consistent with the structure
of the entropic risk functional,
which can be seen by a Taylor expansion of~\eqref{eq:entropic-risk}
with respect to $\theta$, neglecting the difference between
the different controls found by each run.

An explanation for the dependence of the gradient norm
estimate on $\theta$ can be found in Algorithm~\ref{alg:mlsgd}.
Increasing $\theta$ leads to larger batch sizes
$\tset{M_{k,\ell}}_{\ell=1}^{L_k}$ in each iteration
to match the iteration dependent target $\varepsilon_k$.
This in turn reduces the number of optimization steps
that can be performed within the fixed computational time
budget of four hours.
Hence, this also explains the decrease in the total number of
iterations observed in the right plot of
Figure~\ref{fig:theta-experiment}.

\begin{figure}
    \centering
    \includegraphics[width=1.0\linewidth]{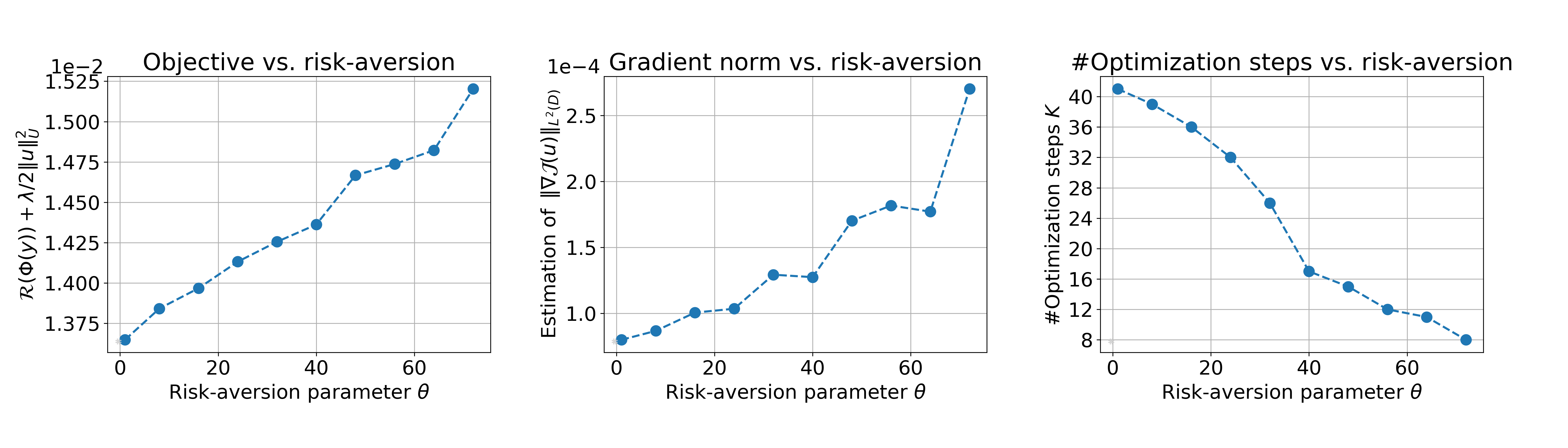}
    \caption{
        Objective~\eqref{eq:cost} (left),
        norm of estimate to gradient~\eqref{eq:exact-gradient} (center),
        total number of optimization steps taken in four hours (right)
        plotted over an increasing risk-aversion parameter $\theta$.
    }
    \label{fig:theta-experiment}
\end{figure}

To illustrate why the batch sizes
$\tset{M_{k,\ell}}_{\ell=1}^{L_k}$
increase with larger $\theta$, we selected three
experiments with $\theta = 8.0$, $\theta = 32.0$, and
$\theta = 72.0$, and plotted the total number of samples,
summed over all iterations $\tset{\sum_k M_{k,\ell}}_{\ell=1}^{L_k}$,
against the discretization level
in the top-left panel of Figure~\ref{fig:exponents}.
Note that all experiments were initialized with the batch sizes
$\tset{128,16,2}$ on unit cubes refined four, five, and six
times, respectively.
Despite the decrease in the total number
of iterations (cf.~Figure~\ref{fig:theta-experiment} on the right),
the total number of samples computed on each
level remains comparable across all values of $\theta$.
Thus, fewer optimization steps are
compensated through larger batches resulting
in the same amount of total samples.

Below this on the left of Figure~\ref{fig:exponents},
we plot the averaged computational time cost
$\mathcal{C}_\ell$ over the discretization levels,
confirming assumption~\eqref{eq:assumption_gamma}.
This quantity is used in~\eqref{eq:optimal-batch-size}
to determine the optimal batch size and is clearly
independent of $\theta$.
This in turn implies that~\eqref{eq:sec-order-sum}
must depend on $\theta$ as it is also influenced by
the estimates~\eqref{eq:decayDenom} and~\eqref{eq:decayNum},
both of which involve constants that depend on $\theta$.

To verify this, we collected measurements from the final
gradient batch across all three experiments and plotted them
in the center column of Figure~\ref{fig:exponents}.
Both estimates,~\eqref{eq:decayDenom} and~\eqref{eq:decayNum},
are clearly satisfied, and the experiments
highlight the dependence of the constants $C_Y$ and $C_X$ on $\theta$.
We note that the measurements of $\beta$ exhibit visible
measurement uncertainty since only a single batch is used
(see also Figure~\ref{fig:nodes-experiment} bottom row in the center).

In the right column of Figure~\ref{fig:exponents},
we plot the estimates
$\|\EE^{\text{MC}}_k[\Delta X_\ell]\|$ and $|\EE^{\text{MC}}_k[\Delta Y_\ell]|$,
which appear in~\eqref{eq:squared-bias-estimate-mlmc}.
These plots confirm the assumptions
in~\eqref{eq:weakerror} and~\eqref{eq:weakerror2},
while again illustrating the dependence of the constants
$\widetilde{C}_Y$ and $\widetilde{C}_X$ on $\theta$.
Since these assumptions enter the algorithm through
the estimator~\eqref{eq:squared-bias-estimate-mlmc},
they strongly influence when an additional mesh refinement
to level seven is triggered.
In particular, larger values of
$\widehat{\mathrm{err}}^{\text{num}}_k$
lead to earlier refinement by appending another level.
From the logged data, we observe that for $\theta = 8.0$
this refinement occurs after approximately
7500 seconds of computation time,
for $\theta = 32.0$ after roughly 4500 seconds,
and for $\theta = 72.0$ an additional level is selected
already after 12 seconds, immediately following
the first iteration.

\begin{figure}
    \centering
    \includegraphics[width=1.0\linewidth]{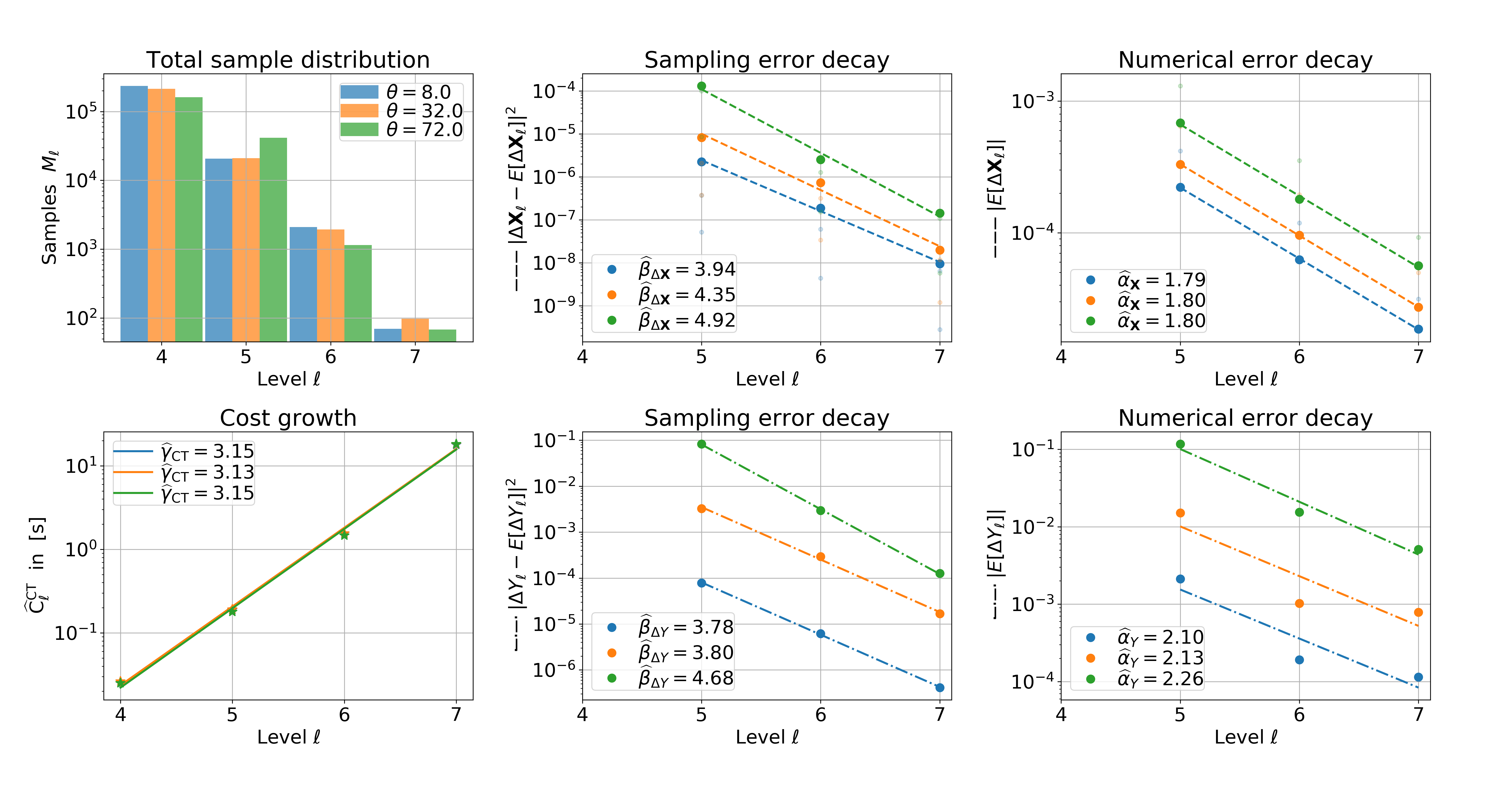}
    \caption{
        Total number of samples $\tset{\sum_k M_{k,\ell}}_{\ell=1}^{L_k}$ (top left),
        verification of~\eqref{eq:assumption_gamma} (bottom left),
        verification of~\eqref{eq:decayDenom} and \eqref{eq:decayNum} (center column),
        verification of~\eqref{eq:weakerror} and \eqref{eq:weakerror2} (right column)
        for three different risk-aversion parameters $\theta \in \set{8.0, 32.0, 72.0}$
    }
    \label{fig:exponents}
\end{figure}
    \subsection{Experiments on node scaling}
\label{subsec:experiments-on-node-scaling}

To compensate for the deterioration of the numerical
error estimates (cf.~the right column of
Figure~\ref{fig:exponents}), the sampling error estimates
(cf.~the center column of Figure~\ref{fig:exponents}),
and the gradient estimates
(cf.~the center plot of Figure~\ref{fig:theta-experiment})
as $\theta$ increases, we keep the total computational
time fixed at four hours while increasing the available
computational resources.
More precisely, we successively
double the number of CPUs employed, starting from
$P = 64$ up to $P = 2048$, to solve the optimization problem,
focusing on the case $\theta = 40$.
Figure~\ref{fig:nodes-experiment} illustrates this scaling experiment,
using blue for $P=128$, orange for $P=512$ and green for $P=2048$,
while experiments using $P \in \set{64, 256, 1024}$
are colored in gray in the lower
right plot to avoid overloaded figures.

The plot on the top left again shows the total number of samples on each level.
Clearly, the additional computational power is used to explore more samples
and, for the case $P \geq 512$, also an additional level.
This observation is further reflected in the bottom-left panel of
Figure~\ref{fig:nodes-experiment}, where we illustrate
the final multiindex set used in the optimization.
As explained in Subsection~\ref{subsec:batch-parallelization},
equation~\eqref{eq:optimal-parallelization} enables
larger batches when more processing capacity is available.
Note that the illustrated multiindex set changes since
$\set{M_{k,\ell}}_{\ell=1}^{L_k}$ varies by~\eqref{eq:optimal-batch-size}
throughout the optimization and that we only show the one
used in the final step.

The increased computational capacity enables
the exploration of more samples and additional levels,
thereby reducing the uncertainty in the final
objective~\eqref{eq:cost} without substantially
changing its value.
This can be observed in the center plot of the
top row of Figure~\ref{fig:nodes-experiment},
where we depict~\eqref{eq:cost}
with error bars based on the square root
of~\eqref{eq:mse-estimate-mlmc}.
The results confirm two points:
first, the control update and parallelization
strategies described in
Subsections~\ref{subsec:batch-parallelization}
work as intended, since otherwise different
numbers of CPUs could lead to different objective values,
a phenomenon commonly encountered in the
development process of the method;
second, the error estimates decrease with increasing computational
resources, as indicated by the shrinking error bars.

The plot on the top right of
Figure~\ref{fig:nodes-experiment}
further supports this observation.
There, we plot the norm of the estimated
gradient~\eqref{eq:gradient_estimator}
against the total computing time in a log--log scale,
omitting the initial optimization steps in the preasymptotic regime,
and demonstrating that more CPUs indeed lead to smaller gradient norms.
By Theorem~\ref{thm:complexity_gradient_estimation},
the inverted complexity estimate~\eqref{eq:cost_asymptotic}
yields $\varepsilon \lesssim_{\theta} \mathcal{C}^\delta$
with $\delta = \min \tset{\tfrac{1}{2}, \tfrac{\alpha}{2 \alpha + (\gamma - \beta)}}$,
excluding the case $\beta = \gamma$.
Since this relation must hold for every iteration $k$,
and since $\varepsilon_k$ controls the gradient norm,
the observed behavior confirms~\eqref{eq:cost_asymptotic},
with an estimated convergence rate of
$\delta \approx 0.33$ obtained from a least-squares fit.
To explain why we observe only
$\delta \approx 0.33$ rather than
$\delta = 0.5$ (as in two-dimensional test cases not illustrated in this paper),
we refer to the two remaining plots in the bottom row.

The center plot in the bottom row of
Figure~\ref{fig:nodes-experiment}
may explain the reduced convergence rate.
Here, we illustrate the estimated exponents
$\alpha$, $\beta$, and $\gamma$ over the iteration index $k$.
While the estimates for $\alpha$ and $\gamma$
remain largely stable throughout the optimization,
with only minor deflections when additional
levels are added or communication indices are switched,
the estimates for $\beta$ are considerably noisier.
For some iterates, the decay of the sampling error
is smaller than the corresponding cost increase $\gamma$,
such that we ultimately end up in the regime
$\gamma > \beta$, leading to reduced convergence rates in $\delta$.
This observation is also consistent with classical MLMC
theory, where achieving a convergence rate of
$\delta = 0.5$ is more difficult in the three-dimensional case
(cf.~\cite{charrier2013finite} for further discussion).

Lastly, the plot in the bottom right shows again the
gradient norm of the last iterate against the used number of CPUs.
While we can not reproduce the same convergence rate $\delta$
by doubling the amount of CPUs as we would, if we double the amount of time,
diminishing returns can not be observed within the span up to $P=2048$ yet.

\begin{figure}
    \centering
    \includegraphics[width=1.0\linewidth]{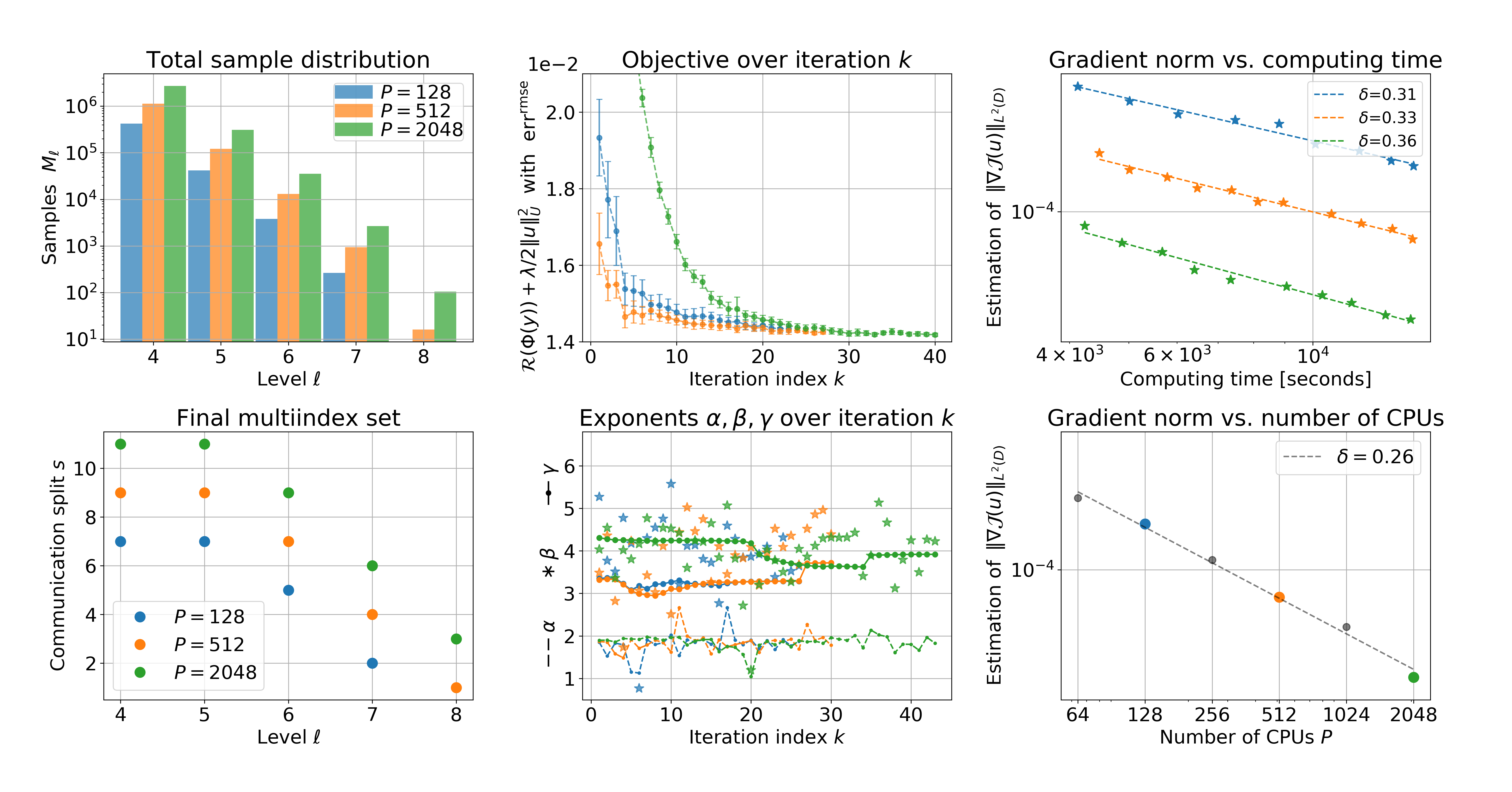}
    \caption{
        Total number of samples $\tset{\sum_k M_{k,\ell}}_{\ell=1}^{L_k}$ (top left),
        estimated objective~\eqref{eq:cost} with root of MSE error bars~\eqref{eq:mse-estimate-mlmc}
        plotted over the iteration $k$ (top center),
        norm of estimated gradient~\eqref{eq:gradient_estimator} plotted over the total computing time (top right),
        communication split~\eqref{eq:optimal-parallelization} over the used levels in the final iteration (bottom left),
        verification of~\eqref{eq:decayNum}-\eqref{eq:assumption_gamma} over all iterations (bottom center),
        norm of estimated gradient~\eqref{eq:gradient_estimator} of the last iteration
        over the used number of processing units (bottom right).
    }
    \label{fig:nodes-experiment}
\end{figure}
    \section{Outlook and Conclusion}
\label{sec:outlook-and-conclusion}

We have demonstrated that MLSGD is applicable to
risk-averse optimization, particularly in computationally demanding settings
with high risk-aversion and three-dimensional PDE constraints.
To this end, we developed a new convergence analysis,
incorporating a multilevel gradient estimator for the entropic risk,
as well as new parallelization strategies for the control updates,
achieving minimal communication overhead and multilevel complexity
results in theory and practice.

Future work may target other PDE constraints,
e.g., acoustic wave equations as in~\cite{baumgarten2024fully},
or risk measures such as conditional value at risk~\cite{rockafellar2000optimization},
as in~\cite{pieraccini2025adaptive},
supported by additional importance sampling techniques.
One major bottleneck for scaling to even larger problems
is the increased memory footprint.
Our investigations show that, in particular,
the multigrid preconditioner, in combination with the multiindex
data structure, leads to elevated memory usage.
Offloading the algebraic workload associated with solving
the state and adjoint equations to GPUs, e.g., by using Gingko~\cite{anzt2022ginkgo},
is a promising way to alleviate this memory constraint.
Furthermore, as the numerical results illustrate,
the overall procedure is limited in complexity
by sampling rather than by the optimization itself.
Hence, improving the sampling complexity,
e.g., with quasi-Monte Carlo methods as in~\cite{guth2021quasi, guth2024parabolic},
is a promising direction to further improve the convergence rate.
    \section*{Acknowledgements}
Parts of the presented work were developed during the
Junior Trimester Program \textit{Computational multifidelity, multilevel, and multiscale methods}
funded by the Deutsche Forschungsgemeinschaft
(DFG, German Research Foundation)
under Germany's Excellence Strategy – EXC-2047/2 – 390685813.

The authors gratefully acknowledge the computing time provided
on the high-performance computer HoreKa by the
National High-Performance Computing Center at KIT (NHR@KIT).
This center is jointly supported by the Federal Ministry of Education
and Research and the Ministry of Science, Research
and the Arts of Baden-Württemberg,
as part of the National High-Performance Computing (NHR)
joint funding program.
HoreKa is partly funded by the German Research Foundation (DFG).
TV has been partially supported by the INdAM-
GNCS project GNCS 2026 - CUP\_E53C25002010001.

The authors are grateful to Simon Weissmann for carefully
reading an earlier version of this manuscript and for bringing
to our attention an issue in the original proof,
which led to the present formulation of the arguments.
    \bibliographystyle{ieeetr}
    \bibliography{bibliography}

\end{document}